\documentclass[12pt]{article}
\usepackage{fullpage,amsthm,graphics,hyperref,bm,xcolor,verbatim,amssymb,
subcaption, amsmath, array, tabularx, comment}
\usepackage[normalem]{ulem}
\usepackage{tikz}
\usetikzlibrary{decorations.pathmorphing}

\usepackage{hyperref}
\hypersetup{colorlinks=true,linkcolor=black,anchorcolor=blue,citecolor=black,filecolor=blue,urlcolor=blue,bookmarksnumbered=true,pdfview=FitB}

\tikzstyle{uStyle}=[shape = circle, minimum size = 20pt, inner sep =2.5pt, outer sep = 0pt, draw, fill=white]
\tikzstyle{mystyle}=[shape = rectangle, minimum size = 20pt, inner sep =2.5pt, outer sep = 0pt, draw, fill=white]

\newenvironment{clmproof}[1]{\par\noindent\underline{Proof.}\space#1}{\leavevmode\unskip\penalty9999\hbox{}\nobreak\hfill\quad\hbox{$\diamondsuit$}\smallskip}

\newtheorem{lem}{Lemma}
\newtheorem{cor}[lem]{Corollary}
\newtheorem{conj}{Conjecture}
\newtheorem{thm}{Theorem}
\newtheorem*{mainthm}{Main Theorem}
\newtheorem*{keylem}{Key Lemma}

\newtheorem*{thmA}{Theorem A}
\theoremstyle{definition}

\newtheorem{prob}{Problem}
\newtheorem{defn}{Definition}
\newtheorem{clm}{Claim}

\newtheorem{example}[conj]{Example}

\RequirePackage{marginnote,hyperref}
\newcommand{\aside}[1]{\marginnote{\scriptsize{#1}}[0cm]}
\newcommand{\aaside}[2]{\marginnote{\scriptsize{#1}}[#2]}
\newcommand\Emph[1]{\emph{#1}\aside{#1}}
\newcommand\EmphE[2]{\emph{#1}\aaside{#1}{#2}}

\newcommand\LL{\mathcal{L}}

\title{Kempe Equivalent List Colorings}
\def\vph{\varphi}
\def\erdos{Erd\"{o}s}
\def\LL{\mathcal{L}}
\def\BB{\mathcal{B}}
\def\AA{\mathcal{A}}
\def\HH{\mathcal{H}}

\def\Z{\mathbb{Z}}

\def\dbox{\!\!\!\!\!\qed\,}
\author{Daniel W. Cranston\thanks{%
Department of Computer Science, Virginia Commonwealth
University, Richmond, VA, USA;
\texttt{dcranston@vcu.edu}
} 
\and Reem Mahmoud\thanks{%
Department of Mathematics and Applied Mathematics, Virginia Commonwealth
University, Richmond, VA, USA; 
\texttt{mahmoudr@vcu.edu}
}
}
\begin{document}
\maketitle

\begin{center}
\vspace{-.2in}
\textit{\small{Dedicated to the memory of Landon Rabern.}}
\end{center}

\begin{abstract}
An $\alpha,\beta$-Kempe swap in a properly colored graph interchanges the colors
on some component of the subgraph induced by colors $\alpha$ and $\beta$.
Two $k$-colorings of a graph are $k$-Kempe equivalent if we can form one from
the other by a sequence of Kempe swaps (never using more than $k$ colors).  Las
Vergnas and Meyniel showed that if a graph is $(k-1)$-degenerate, then each
pair of its $k$-colorings are $k$-Kempe equivalent.  Mohar conjectured the same
conclusion for connected $k$-regular graphs.  This was proved for $k=3$ by
Feghali, Johnson, and Paulusma (with a single exception $K_2\dbox K_3$, also
called the 3-prism) and for $k\ge 4$ by Bonamy, Bousquet, Feghali,
and Johnson.

In this paper we prove an analogous result for list-coloring.  For a
list-assignment $L$ and an $L$-coloring $\vph$, a Kempe swap is called $L$-valid for
$\vph$ if performing the Kempe swap yields another $L$-coloring.  Two
$L$-colorings are called $L$-equivalent if we can form one from the other by a sequence
of $L$-valid Kempe swaps.  Let $G$ be a connected $k$-regular graph with $k\ge
3$ and $G\ne K_{k+1}$.  We prove that if $L$ is a $k$-assignment, then all $L$-colorings are
$L$-equivalent (again excluding only $K_2\dbox K_3$).  
Further, if $G\in\{K_{k+1},K_2\dbox K_3\}$, $L$ is a $\Delta$-assignment, but
$L$ is not identical everywhere, then all $L$-colorings of $G$ are $L$-equivalent.
When $k\ge 4$, the proof is completely self-contained, implying an
alternate proof of the result of Bonamy et al.

Our proofs rely on the following key lemma, which may be of independent
interest.  Let $H$ be a graph such that for every degree-assignment $L_H$ all
$L_H$-colorings are $L_H$-equivalent.  If $G$ is a connected graph that
contains $H$ as an induced subgraph, then for every degree-assignment $L_G$ for
$G$ all $L_G$-colorings are $L_G$-equivalent.
\end{abstract}

\section{Introduction}
An $\alpha,\beta$-Kempe swap in a proper (partial) coloring of a graph
interchanges the colors on some component of the subgraph induced by colors
$\alpha$ and $\beta$. 
Thus, a Kempe swap always yields another proper (partial) coloring.
Kempe swaps were introduced in the late 1800s in an
attempt to prove the 4 Color Theorem.  Given a partial proper coloring, we aim
to perform one or more Kempe swaps to reach another partial proper coloring
that can be extended to an additional vertex.  By repeating this process, we
eventually color the whole graph, hopefully.  This initial effort to prove the
4 Color Theorem was unsuccessful.  However, Heawood salvaged the
idea and used it to prove the 5 Color Theorem.  Kempe swaps are particularly
useful in the study of edge-coloring, which is equivalent to coloring line
graphs.\footnote{This is because in a line graph each subgraph recolored by a Kempe swap
must be a path or an even length cycle.  So it is easier to understand how the
coloring changes after a single Kempe swap.}

Although Kempe swaps are still used as a tool in constructing proper colorings,
they have also sparked a new interest.  Two proper $k$-colorings of a graph $G$
are \emph{(Kempe) $k$-equivalent} if we can form one from the other by a sequence
of Kempe swaps (never using more than $k$ colors).  When all $k$-colorings of a
graph are $k$-equivalent, we have a
natural way to sample a $k$-coloring randomly.  Starting from an arbitrary
$k$-coloring, we perform a sequence of Kempe swaps, each chosen randomly from
those available in the current coloring, until our current $k$-coloring is
``almost'' equally likely to be any one of the $k$-colorings of $G$.  To
formalize all this, and clarify some technical details, we would use Markov
chains (which we avoid in this paper).  This problem and many related ones have
been studied widely under the name
\emph{reconfiguration}~\cite{Nishimura, Heuvel}.  Given two instances of some type of object, we
ask whether we can ``reconfigure'' one to be the other, using a sequence of
reconfiguration steps.  In our case, the objects are $k$-colorings and the
reconfiguration steps are Kempe swaps.

In the late 1970s, Meyniel~\cite{meyniel} proved that if $G$ is a planar graph,
then all of its 5-colorings are 5-equivalent; see
also~\cite{meyniel-strengthened}.  This was extended by Las Vergnas
and Meyniel~\cite{lVM}, who
proved the same conclusion for all $K_5$-minor-free graphs.  They also showed
that if $G$ is $(k-1)$-degenerate, then all of its $k$-colorings are
$k$-equivalent; Lemma~\ref{degen-lem} below generalizes this result to list coloring.
Mohar conjectured that if $G$ is connected and $k$-regular, then all of its
$k$-colorings are $k$-equivalent.  This is a natural next step, since the
sparsest graphs that are not $(k-1)$-degenerate are $k$-regular.  
Mohar's Conjecture was proved for $k=3$ by Feghali, Johnson, and
Paulusma~\cite{FJP} (with a single exception $K_2\dbox K_3$) and for $k\ge 4$
by Bonamy, Bousquet, Feghali, and Johnson~\cite{BBFJ}; this was reproved in a
stronger form in~\cite{BDL-D}.
In this paper our Main Theorem (stated on the next page)
is an analogous result for list-coloring.

A \emph{block} in a graph is a maximal 2-connected subgraph.
A \emph{list-assignment}\aside{block, list-assignment} $L$ for a graph $G$ gives each $v\in V(G)$ a
list $L(v)$ of allowable colors.
Given $L$, an \EmphE{$L$-coloring}{4mm} is a proper coloring $\vph$
such that $\vph(v)\in L(v)$ for all $v\in V(G)$.  For an $L$-coloring $\vph$, a
Kempe swap is \emph{$L$-valid}\aaside{\mbox{$L$-valid},~$L$-equivalent}{-2.5mm}
(for $\vph$) if performing the Kempe swap on $\vph$ yields
another $L$-coloring.  Two $L$-colorings are \emph{$L$-equivalent} if we can
form one from the other by a sequence of $L$-valid Kempe swaps.  
A \emph{$k$-assignment}
(resp.~\emph{$f$-assignment})\aside{\mbox{$k$-/$f$-/degree-} \mbox{assignment}} to a graph $G$ is a list
assignment $L$ such that $|L(v)|=k$ (resp.~$|L(v)|=f(v)$), for each $v\in V(G)$,
where $k$ is some constant (resp.~$f:V(G)\to\mathbb{Z}^+$).  
A \emph{degree-assignment} to $G$ is a list assignment $L$ such that
$|L(v)|=d(v)$ for each $v\in V(G)$.  Given $L$, a graph $G$ is
\emph{$L$-swappable} if it has some $L$-coloring\footnote{By requiring at least
one $L$-coloring, we can state many of our results more cleanly.} and all of
its $L$-colorings are $L$-equivalent.
A graph $G$ is \emph{degree-swappable}\aaside{\mbox{$L$-/$k$-/$f$-} /degree-
swappable}{-8mm}
(resp.~\emph{$k$-swappable} or \emph{$f$-swappable}) if $G$ is
$L$-swappable for every degree-assignment $L$ (resp.~$k$-assignment or
$f$-assignment $L$).
To illustrate some of the key ideas in this paper,
we now reprove a helpful lemma of Las Vergnas and Meyniel~\cite{lVM},
generalized to the context of list-coloring.

\begin{lem}
\label{degen-lem}
Let $G$ be a connected graph, let $L$ be a list-assignment for $G$, and fix $v\in V(G)$ with
$|L(v)|>d(v)$.  Let $G':=G-v$ and let $L'$ denote $L$ restricted to $G'$.
If $G'$ is $L'$-swappable, then $G$ is $L$-swappable.
\end{lem}
\begin{proof}
Assume $G'$ is $L'$-swappable.
Let $\vph_1$ and $\vph_2$ denote $L$-colorings of $G$, and let $\vph'_1$ and
$\vph'_2$ denote their restrictions to $G'$.  Since $G'$ is $L'$-swappable,
there exist $L'$-colorings $\psi'_0,\ldots,\psi'_t$ of $G'$ such that
$\psi'_0=\vph'_1$ and $\psi'_t=\vph'_2$ and $\psi'_i$ differs from
$\psi'_{i-1}$ by a single $L'$-valid Kempe swap, for each $i\in [t]$.  Now we
extend each $\psi'_i$ to an $L$-coloring $\psi_i$ of $G$ such that $\psi_i$ and
$\psi_{i-1}$ are $L$-equivalent.  Suppose that $\psi'_i$ differs from
$\psi'_{i-1}$ by an $\alpha,\beta$-swap at $v_i$, for some $v_i\in V(G)$ and
some colors $\alpha$ and $\beta$.  If $\psi_{i-1}(v)\notin\{\alpha,\beta\}$ or
if $v$ is not in the same $\alpha,\beta$-component of $\psi_{i-1}$ as $v_i$,
then we form $\psi_i$ from $\psi_{i-1}$ by performing the same
$\alpha,\beta$-swap at $v_i$.  This approach also works if
$\{\alpha,\beta\}\subseteq
L(v)$ and $v$ is in the same $\alpha,\beta$-component as $v_i$, but $v$ has
degree 1 in that component.  So suppose that $v$ is in the same
$\alpha,\beta$-component as $v_i$, but either $|L(v)\cap \{\alpha,\beta\}|=1$
or $v$ has degree at least 2 in that $\alpha,\beta$-component.  Since
$|L(v)|>d(v)$, there exists $\gamma\in L(v)$ that is unused by $\psi_{i-1}$
on the closed neighborhood of $v$.  We first recolor $v$ with $\gamma$, and
then perform the $\alpha,\beta$-swap at $v_i$.  By induction on $i$, this
gives an $L$-coloring $\psi_t$ that restricts to $\psi'_t$.  Lastly, if
$\psi_t(v)\ne \vph_2(v)$, then we recolor $v$ with $\vph_2(v)$.
\end{proof}

\begin{cor}
\label{degen-cor}
A graph $G$ is $L$-swappable whenever there exists a vertex order in which each
vertex $x$ is preceded by fewer than $|L(x)|$ neighbors.  In particular, $G$ is
$L$-swappable when $G$ is $(k-1)$-degenerate and $L$ is a $k$-assignment.  This
includes the special case that $k:=\Delta$ and $G$ is connected, but not
regular.
\end{cor}
\begin{proof}
We prove the first statement by induction on $|V(G)|$.  The base case $|V(G)|=1$
holds trivially.  The induction step holds by Lemma~\ref{degen-lem}, taking
$v$ to be the final vertex in the order.  The second statement follows from the
first, using any order that witnesses that $G$ is $(k-1)$-degenerate.
Finally, the third statement obviously follows from the second, when we order the
vertices by non-increasing distance from some vertex of degree less than $k$.
\end{proof}

Now we can state our Main Theorem and outline its proof.

\begin{mainthm}
If $G$ is a connected graph with $\Delta\ge 3$ and $G\notin\{K_2\dbox K_3,
K_{\Delta+1}\}$, then $G$ is $\Delta$-swappable.
If $G\in\{K_2\dbox K_3,K_{\Delta+1}\}$ and $L$ is a $\Delta$-assignment that is
not identical everywhere, then $G$ is $L$-swappable.
\end{mainthm}

Note that the interesting case in our Main Theorem is when $G$ is regular, since
otherwise the result is included in Corollary~\ref{degen-cor}.
A crucial step in proving the Main Theorem is verifying the following Key Lemma.
Its proof mirrors that of Lemma~\ref{degen-lem}.

\begin{keylem}
If $H$ is degree-swappable and $G$ is a connected graph containing $H$ as an
induced subgraph, then $G$ is degree-swappable.
\end{keylem}

In view of the Key Lemma, we have a natural plan to prove the Main Theorem. (i)
Compile a collection $\HH$ of known degree-swappable graphs. (ii) Show that if
$G$ is connected with $\Delta\ge 3$ and $G\notin\{K_2\dbox K_3,
K_{\Delta+1}\}$, then $G$ contains as an induced subgraph some $H\in \HH$.  
(For brevity, we omit from this sketch the details of handling $K_2\dbox K_3$
and $K_{\Delta+1}$, but they are not hard.) In fact, \erdos, Rubin, and
Taylor~\cite{ERT} used a similar approach to characterize all degree-choosable
graphs.  They showed that if $G$ is connected and not a Gallai tree, then $G$
contains as an induced subgraph an even length cycle with at most one
chord\footnote{The general case of this result easily reduces to the case when
$G$ is 2-connected, which is known as Rubin's Block Lemma.
For a shorter proof, see Section~9 of~\cite{brooks-survey}.},
which we call a \Emph{good cycle}.  It is easy to check that every good cycle is
degree-choosable.  A \Emph{Gallai tree} is a connected graph in which each block
is an odd cycle or a clique.  These results on good cycles imply that every
connected graph is degree-choosable unless it is a Gallai tree; this implication
uses an analogue of our Key Lemma that holds for degree-choosability.\footnote{To
prove this analogue, we greedily color the vertices of $G\setminus H$ in order of
non-increasing distance from $H$.  Afterward, we can extend this coloring to $H$
precisely because $H$ is degree-choosable.}  So perhaps we might hope to even
characterize
degree-swappable graphs.  But when we take this approach, we quickly find many
degree-choosable graphs that are not degree-swappable.

\begin{example}
\label{ex1}
Denote the vertices of an $n$-cycle $C_n$ by $v_1,\ldots,v_n$.  Let
$L(v_i)=\{i,i+1\} \bmod{n}$.  See the left side of Figure~\ref{non-swappable-fig}.
Note that $C_n$ has two $L$-colorings:
(1) $\vph(v_i)=i$, for all $v_i$, and (2) $\vph(v_i)=i+1$, for all $v_i$.  However,
$|L(v_i)\cap L(v_j)|\le 1$ for all distinct $i,j$.  Thus, neither $\vph_1$ nor
$\vph_2$ admits any $L$-valid Kempe swap.  So $C_n$ is not degree-swappable.

\begin{figure}[!h]
\centering
\begin{tikzpicture}[scale=.9]
\clip (-1.0,-12.0) rectangle (12.0,-4.4);
\tikzstyle{uStyle}=[shape = circle, minimum size = 4.5pt, inner sep = 0pt,
outer sep = 0pt, draw, fill=white]
\tikzset{every node/.style=uStyle}

\def\myspace{.58}
\renewcommand{\ULthickness}{1.4pt}%

\begin{scope}[scale=.64, xscale=.9, rotate=-90]
\begin{scope}[xshift=3.4in]
\draw [thick]
(0,0) node (A) {} --
(0,2) node (B) {} --
(0,4) node (C) {} --
(2,4) node (D) {} --
(2,2) node (E) {} --
(2,0) node (F) {} --
(0,0);
\foreach \where/\xshift/\lab in 
{A/-1/{1}~2,B/-1/{2}~3,C/-1/{3}~4,
D/1/{4}~5,E/1/{5}~6,F/1/{6}~1}
\draw (\where) ++ (\myspace*\xshift,0) node[draw=none,shape=rectangle] {{\lab}};

\draw [thick]
(0,0) node (A) {} --
(0,2) node (B) {} --
(0,4) node (C) {} --
(2,4) node (D) {} --
(2,2) node (E) {} --
(2,0) node (F) {} --
(0,0);
\foreach \where/\xshift/\lab in 
{A/-1/\uline{1}~2,B/-1/\uline{2}~3,C/-1/\uline{3}~4,
D/1/\uline{4}~5,E/1/\uline{5}~6,F/1/\uline{6}~1}
\draw (\where) ++ (\myspace*\xshift,0) node[draw=none,shape=rectangle] {{\lab}};

\begin{scope}[xshift=2.35in]
\draw [thick]
(0,0) node (A) {} --
(0,2) node (B) {} --
(0,4) node (C) {} --
(2,4) node (D) {} --
(2,2) node (E) {} --
(2,0) node (F) {} --
(0,0);
\foreach \where/\xshift/\lab in 
{A/-1/1~\uline{2},B/-1/2~\uline{3},C/-1/3~\uline{4},
D/1/4~\uline{5},E/1/5~\uline{6},F/1/6~\uline{1}}
\draw (\where) ++ (\myspace*\xshift,0) node[draw=none,shape=rectangle] {{\lab}};
\end{scope}

\end{scope}
\end{scope}
\begin{scope}[scale=.655, thick, xshift=3in, yshift=-3.30in]
\tikzstyle{uStyle}=[shape = circle, minimum size = 4.5pt, inner sep = 0pt,
outer sep = 0pt, draw, fill=white]
\tikzset{every node/.style=uStyle}

\draw (5,4) node[white] {};

\draw (0,0) node (A) {} 
--++ (2,0) node (B) {} 
--++ (2,0) node (C) {} 
--++ (2.6,0) node (D) {} 
--++ (2,0) node (E) {} 
--++ (0,-2) node (F) {} 
--++ (-2,0) node (G) {} 
-- (D) -- (F) (G) -- (E)
(C) --++ (240:2cm) node (H) {}
--++ (2,0) node (I) {} -- (C)
(B) --++(240:2cm) node (J) {}
-- (A);
\foreach \where/\yshift/\lab in 
{A/1/{\uline{0}\,1\,2},B/1/{\uline{1}\,2\,3},C/1/{\uline{3}\,4\,5\,6},
D/1/{\uline{6}\,7\,8\,9},E/1/{~~\uline{7}\,8\,9}
,F/-1/{~~7\,\uline{8}\,9\,0}
,G/-1/{7\,8\,\uline{9}}
,H/-1/{4\,\uline{5}}
,I/-1/{\uline{4}\,5}
,J/-1/{1\,\uline{2}}
}
\draw (\where) ++ (0,\myspace*\yshift) node[draw=none,shape=rectangle] {{\lab}};

\draw (F) edge[out = 358, in = 240, looseness = 2.3] (A);

\begin{scope}[yshift=-2.3in]

\draw (0,0) node (A) {} 
--++ (2,0) node (B) {} 
--++ (2,0) node (C) {} 
--++ (2.6,0) node (D) {} 
--++ (2,0) node (E) {} 
--++ (0,-2) node (F) {} 
--++ (-2,0) node (G) {} 
-- (D) -- (F) (G) -- (E)
(C) --++ (240:2cm) node (H) {}
--++ (2,0) node (I) {} -- (C)
(B) --++(240:2cm) node (J) {}
-- (A);
\foreach \where/\yshift/\lab in 
{A/1/{0\,\uline{1}\,2},B/1/{1\,2\,\uline{3}},C/1/{3\,4\,5\,\uline{6}},
D/1/{6\,\uline{7}\,8\,9},E/1/{~~7\,8\,\uline{9}}
,F/-1/{~~7\,8\,9\,\uline{0}}
,G/-1/7~\uline{8}~9
,H/-1/4~\uline{5}
,I/-1/\uline{4}~5
,J/-1/1~\uline{2}
}
\draw (\where) ++ (0,\myspace*\yshift) node[draw=none,shape=rectangle] {{\lab}};

\draw (F) edge[out = 358, in = 240, looseness = 2.3] (A);
\end{scope}
\end{scope}
\end{tikzpicture}
\caption{Left: A 6-cycle and a 2-assignment showing that it is not
degree-swappable.  Right: Another ``Gallai tree plus edge'' and a degree-assignment showing
that it is not degree-swappable.\label{non-swappable-fig}}
\end{figure}
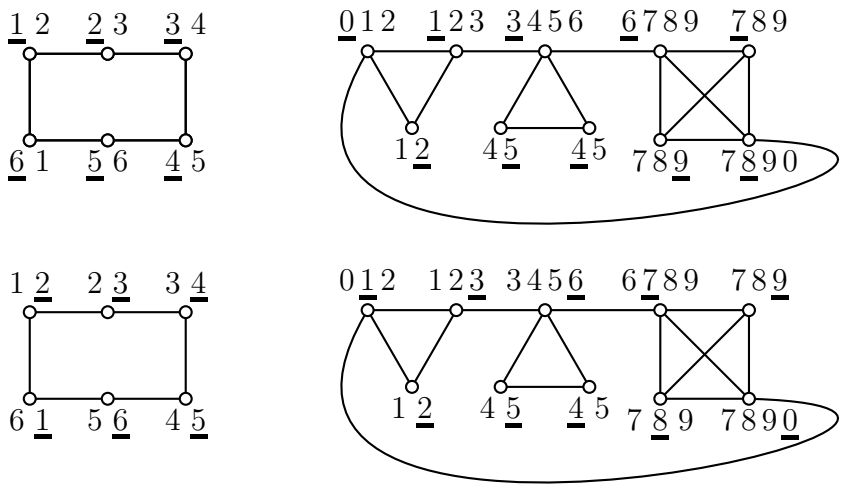

More generally, let $G$ be any Gallai tree.  For each block $B_i$ of $G$, let
$d_i:=d_{B_i}(v)$ for all $v\in V(B_i)$.  Assign to each block $B_i$ a list
$L_i$ of size $d_i$ such that $L_{i_1}\cap L_{i_2}=\emptyset$ whenever $i_1\ne
i_2$.  Let $L(v):=\cup_{B_i\ni v}L_i$.  Now $G$ is not $L$-choosable, as is easy
to verify by induction on its number of blocks.  (This is the standard
construction showing that Gallai trees are not degree-choosable.)  Form $G'$
from $G$ by adding some edge $xy$ such that $x,y\in V(G)$ but $xy\notin E(G)$, $x$ and $y$ are in distinct blocks, and $G'$ is not a Gallai tree.
Let $L'(x):=L(x)\cup\{\alpha\}$, $L'(y):=L(y)\cup\{\alpha\}$, and $L'(z):=L(z)$
for all $z\in V(G)\setminus\{x,y\}$, and some $\alpha\notin \bigcup_{v\in
V(G)}L(v)$.  
See the right side of Figure~\ref{non-swappable-fig}.
Note that $G'$ is $L'$-choosable (by the result of \erdos, Rubin, and
Taylor~\cite{ERT}, since $G'$ is not a Gallai tree).  
Furthermore, $G'$ has some
$L'$-coloring $\vph_1$ with $\vph_1(x)=\alpha$ and has
some other $L'$-coloring
$\vph_2$ with $\vph_2(y)=\alpha$.  Moreover, since $G$ is not $L$-choosable,
every $L'$-coloring of $G'$ uses $\alpha$ on either $x$ or $y$.  Note that
$L'(x)\cap L'(y)=\{\alpha\}$, so no $L$-valid Kempe swap can move $\alpha$ from
$x$ to $y$, or vice versa.  Thus, $\vph_1$ and $\vph_2$ are not $L$-equivalent.
This implies that $G'$ is not degree-swappable, even though, as noted above, $G'$ is degree-choosable.
\end{example}

Note that Example~\ref{ex1} includes an even cycle with a single chord
$e$ whenever the two cycles containing $e$ each have odd length (deleting any edge besides the chord gives a Gallai tree).
Thus, many
good cycles are not degree-swappable.  In fact, we have discovered further
graphs that are not degree-swappable, and we do not yet have a conjectured
description of all such graphs.  So we suggest the following problem.

\begin{prob}
Characterize all degree-swappable graphs.
\end{prob}

To prove the Main Theorem, we split into two cases: (i) connectivity at most 3 and
(ii) connectivity at least 4.  In the first case, which takes most of the work, we use
a small vertex cut to show that $G$ contains an induced subgraph $H$ from a 
family of known degree-swappable graphs.  In the second case, 
the higher connectivity allows us to more explicitly construct a sequence of
$L$-valid Kempe swaps to transform any $L$-coloring $\vph_1$ into any other
$\vph_2$.
The rest of the paper is organized as follows.
In Section~\ref{swappability-lem-sec} we prove the Key Lemma, as well as a few
other helpful results on swappability.
In Section~\ref{degree-swap-sec} we compile a family of known degree-swappable
graphs.  Finally, in Section~\ref{main-proof-sec} we prove the Main Theorem.

\section{Swappability Lemmas}
\label{swappability-lem-sec} 
In this section, we prove a number of lemmas about swappability.  More
precisely, each lemma considers a graph $G$ and a list assignment $L$ and
identifies a set of $L$-colorings of $G$ that are pairwise $L$-equivalent. Each of these
results can be viewed as extending Lemma~\ref{degen-lem}.

\begin{lem}
(a) If $|L(v)|\ge d(v)$ for all $v$ and $|L(w)|>d(w)$ for some
$w$, then $G$ is $L$-swappable (assuming that $G$ is connected).  
(b) If $x\in V(G)$ and $G-x$ is connected, then the same result holds even if we 
only require $|L(x)|\ge 1$.
\label{extra-lem}
\end{lem}
\begin{proof}
The first statement follows directly from Corollary~\ref{degen-cor}.  We order
the vertices by non-increasing distance from $w$; thus, every vertex other than
$w$ has a neighbor later in the order.  For the second statement, we simply
put $x$ first in the order, and apply the previous result to $G-x$, which is
still connected.
\end{proof}

Fix a graph $G$, a vertex $v\in V(G)$, a list assignment $L$ for $G$, and a
color $\alpha\in L(v)$.  
Let $\LL$ denote the set of all $L$-colorings of $G$.
Let \emph{$\LL_{v,\alpha}$}\aside{$\LL$, $\LL_{v,\alpha}$} denote the set of
$L$-colorings $\vph$ such that $\vph(v)=\alpha$.  
If $\LL'$ is a set of $L$-colorings of $G$ that are pairwise $L$-equivalent, then 
$\LL'$ \Emph{mixes}.  If also $\vph$ is an $L$-coloring that is $L$-equivalent
to some $\vph'\in \LL'$, then we say that $\vph$ \EmphE{mixes with
$\LL'$}{-5mm}; often it is the case that $\vph\notin\LL'$.

\begin{lem}
\label{missing-lem}
Let $G$ be a connected graph such that $w,x\in V(G)$, $wx\in E(G)$, and $G-x$ is
connected.  If $L$ is a degree-assignment for $G$ such that there exists $\alpha\in
L(x)\setminus L(w)$, then $\LL_{x,\alpha}$ is nonempty and mixes.  More
generally, $\cup_{\alpha\in L(x)\setminus L(w)}\LL_{x,\alpha}$ mixes.
\end{lem}
\begin{proof}
We let $G':=G-wx$, let $L'(x):=L(x)\setminus L(w)$, and let $L'(v):=L(v)$ for all $v\in
V(G)\setminus\{x\}$.  Now we apply Lemma~\ref{extra-lem}(b) to $G'$ and $L'$.
\end{proof}

\begin{lem}
\label{common-lem}
Let $G$ be a graph with $v,w_1,w_2\in V(G)$ such that $G-\{w_1,w_2\}$ is connected,
$w_1,w_2\in N(v)$, and $w_1w_2\notin E(G)$.  Fix a degree-assignment  $L$ for $G$.
\begin{enumerate}
\item[(1)] If there exists $\alpha\in L(w_1)\cap L(w_2)$, then
$\LL_{w_1,\alpha}\cap \LL_{w_2,\alpha}$ is nonempty and mixes.
\item[(2)] If there exist $\alpha\in L(w_1)\cap L(w_2)$ and $\beta\in L(w_1)\setminus
L(v)$, then
$(\LL_{w_1,\alpha}\cap \LL_{w_2,\alpha})\cup \LL_{w_1,\beta}$ is nonempty and mixes.
\item[(3)] If there exist $\alpha\in L(w_1)\cap L(w_2)$ and $\beta\in L(w_1)\setminus
L(v)$ and also $N(w_1)=N(w_2)$, then
$\bigcup_{\alpha\in L(w_1)\cap L(w_2)}(\LL_{w_1,\alpha}
\cup\LL_{w_2,\alpha})\cup \bigcup_{\beta\in L(w_1)\setminus L(v)}\LL_{w_1,\beta}$ is nonempty and mixes.
\end{enumerate}
\end{lem}
\begin{proof}
To prove (1), let $G':=G-\{w_1,w_2\}$, let $L'(y):=L(y)\setminus\{\alpha\}$ for all $y\in
N(w_1)\cup N(w_2)$, and let $L'(z):=L(z)$ for all other $z\in V(G)$.  
Note that $|L'(v)|\ge |L(v)|-1>d_G(v)-2=d_{G'}(v)$.  
Thus, we can apply Lemma~\ref{extra-lem}(a) to $G'$ and $L'$.  

Now we prove (2). If $\alpha=\beta$, this holds by Lemma~\ref{missing-lem} 
(with $x:=w_1$ and $w:=v$). 
So assume $\alpha\neq\beta$. Let $G':=G-vw_1-w_2$ and let
$L'(w_1):=\{\alpha,\beta\}$ and $L'(z):=L(z)\setminus\{\alpha\}$ for all $z\in
N(w_2)$ and $L'(z):=L(z)$ otherwise.
Now
$L'$ mixes for $G'$, by Lemma~\ref{extra-lem}(a), with $w:=v$. 
These $L'$-colorings of $G'$ are in bijection with colorings of $G$ in
$\LL_{w_2,\alpha}\cap(\LL_{w_1,\alpha}\cup\LL_{w_1,\beta})$, and each $L'$-valid
Kempe swap in $G'$ maps to an $L$-valid Kempe swap in $G$ that respects this bijection.
Since
$\LL_{w_1,\beta}$ mixes by Lemma~\ref{missing-lem} and
$\LL_{w_1,\alpha}\cap\LL_{w_2,\alpha}$ mixes by (1),  the result follows. (Note
that $L_{w_1,\beta}\cap L_{w_2,\alpha}\ne\emptyset$, since we can color $w_1$
with $\beta$ and color $w_2$ with $\alpha$, and then color $G-\{w_1,w_2\}$
greedily towards $v$.)

Finally, we prove (3).  Consider $\vph\in \LL_{w_1,\alpha}\cup\LL_{w_2,\alpha}$.
If $\vph(w_1)\ne\vph(w_2)$, then we simply recolor $w_1$ or $w_2$ so that they
both use color $\alpha$; this is possible because $N(w_1)=N(w_2)$, so $\alpha$
is unused on $N(w_1)$ (and on $N(w_2)$).  Thus, $\vph$ mixes with
$(\LL_{w_1,\alpha}\cap \LL_{w_2,\alpha})\cup \LL_{w_1,\beta}$ by (2).  Finally, if
there exist distinct 
$\alpha,\alpha'
\in L(w_1)\cap L(w_2)$, 
then we repeat
the argument above with $\alpha'$ in place of $\alpha$.  Similarly, if there
exist distinct 
$\beta,\beta'
\in L(w_1)\setminus L(v)$, 
then we repeat the argument
above with $\beta'$ in place of $\beta$.  This proves (3).
\end{proof}

We will often want to prove that a graph $G$ is $L$-swappable, for some list assignment $L$.
When we want to prove that a graph is degree-swappable, the following lemma significantly 
restricts the possibilities for $L$ that we must consider.

\begin{lem}
Fix a graph $G$, a degree-assignment $L$, and an edge $vw$ such that $G-vw$ is
connected and degree-choosable.  If $|L(v)\cap L(w)|\le 1$, then $G$ is
$L$-swappable.
\label{overlap-lem}
\end{lem}
\begin{proof}
If $|L(v)\cap L(w)|=0$, then $G$ is $L$-swappable if and only if $G-vw$ is $L$-swappable,
and the latter statement holds by Lemma~\ref{extra-lem}(a).  So assume instead that
$|L(v)\cap L(w)|=1$ and that $L(v)\cap L(w)=\{\alpha\}$.  
Form $L_1$ from $L$ by removing $\alpha$ from $L(v)$. 
Form $L_2$ from $L$ by removing $\alpha$ from $L(w)$. 
Form $L_3$ from $L$ by removing $\alpha$ from both $L(v)$ and $L(w)$. 
Note that $G$ is $L_1$-swappable if and only if $G-vw$ is $L_1$-swappable, and
the latter is true by Lemma~\ref{extra-lem}(a).  The same is true for
$L_2$-swappable.  Since $L_3$ is a degree-assignment for $G-vw$, by assumption
$G-vw$ has an $L_3$-coloring $\vph$.  Note that $\vph$ is both an
$L_1$-coloring and an $L_2$-coloring. Thus, $L_1$-colorings mix with
$L_2$-colorings. Since every $L$-coloring of $G$ is either an $L_1$-coloring or
an $L_2$-coloring (or both), we conclude that $G$ is $L$-swappable.
\end{proof}

The rest of this section is dedicated to proving the Key Lemma (from the
introduction).  In fact, we prove a more general version, 
Lemma~\ref{H-lem}.  
\begin{lem}
Fix a graph $G$ and a function $f:V(G)\to \Z^+$.  Let $H$ be an induced subgraph
of $G$ such that $G-H$ is $f$-swappable.  Let $f'(x):=f(x)-(d_G(x)-d_H(x))$ for
all $x\in V(H)$.  If $f'(x)\ge d_H(x)$ for all $x\in V(H)$ and $H$ is
$f'$-swappable, then $G$ is $f$-swappable.
\label{H-lem}
\end{lem}
\noindent To see that Lemma~\ref{H-lem} generalizes the Key Lemma, let
$f(v):=d_G(v)$ for all $v\in V(G)$, and note that $f'(v)=d_H(v)$ for all $v\in
V(H)$ and $G-H$ is $f$-swappable by Lemma~\ref{extra-lem}(a).  (Each component
of $G-H$ has a vertex $w$ with a neighbor in $H$.)

In this paper we only need the version of the lemma from the introduction (when
$f(v)=d(v)$), but the more general version is no harder to prove, and it is useful
elsewhere~\cite{edge-swappable}.  The proof of Lemma~\ref{H-lem} mirrors that
of Lemma~\ref{degen-lem}.  We start with $L$-colorings $\vph_1'$ and $\vph_2'$
of $G\setminus H$ and a sequence of $L$-colorings $\psi_0',\ldots,\psi_t'$,
showing that $\vph_1'$ and $\vph_2'$ are $L$-equivalent.  And we seek to extend
these $L$-colorings of $G\setminus H$ to $L$-colorings of $G$.  Our main
obstacle is the possibility that at some step colors used on $H$ might interfere
with our desired Kempe swap.  So first we prove, for every $L$-coloring
$\vph'$ of $G\setminus H$ and every $L$-valid Kempe swap for $\vph'$, that some
extension of $\vph'$ to an $L$-coloring of $G$ does not interfere.
This notion of ``non-interference'' motivates Definition~\ref{versatile-defn}
and Lemma~\ref{versatile-lem}.

\begin{defn}
\label{versatile-defn}
For a graph $G$ and a list assignment $L$ for $G$, an $L$-coloring $\vph$ of $G$
is \EmphE{$(\alpha,\beta)$-versatile}{-4mm} \emph{at $w$} if $\vph(w)\in\{\alpha,\beta\}$ and
an $(\alpha,\beta)$-swap at $w$ is $L$-valid for $\vph$.
\end{defn}

\begin{lem}
\label{versatile-lem}
Fix a graph $G$, a connected subgraph $H$, and a list assignment $L$ for $V(G)$.
Let $\vph'$ be an $L$-coloring for $G-H$ that is $(\alpha,\beta)$-versatile at a
vertex $w$.  If $|L(v)|\ge d_G(v)$ for all $v\in V(H)$, and $H$ is not a Gallai
tree, then there exists an $L$-coloring $\vph$ of $G$ that extends $\vph'$ such
that $\vph$ is $(\alpha,\beta)$-versatile at $w$.  Further, there exists such an
$L$-coloring $\vph$ with the property that each $(\alpha,\beta)$-component of
$\vph$ contains the vertex set of at most one $(\alpha,\beta)$-component of $\vph'$.
\end{lem}

\begin{proof}
Since $H$ is not a Gallai tree, it contains an induced even cycle, $C$,
with at most one chord.  We first show how to extend $\vph'$ to $G-C$, and then
how to extend it to all of $G$.

A key step is to show that if $|V(H)|\ge 2$ and $x\in V(H)$, then there exists
an $L$-coloring $\vph$ of $G-(H-x)$ that extends $\vph'$ to $x$ and that is
$(\alpha,\beta)$-versatile at $w$.  Suppose that $|V(H)|\ge 2$ and fix $x\in
V(H)$.  When choosing a color for $x$ (for brevity, we denote it by
$\vph'(x)$), to ensure that the resulting extension of $\vph'$ is an
$L$-coloring that is $(\alpha,\beta)$-versatile at $w$, we need to check the
following three properties: (a) $\vph'(x)\ne \vph'(y)$ for all $y\in N(x)\setminus H$,
(b) if $\vph'(x)\in\{\alpha,\beta\}$ and $x$ has a neighbor $y\notin H$ with
$\vph'(y)\in\{\alpha,\beta\}$, then $\{\alpha,\beta\}\subseteq L(x)$, and (c) if
$\vph'(x)\in\{\alpha,\beta\}$, then $x$ 
has at most one neighbor $y\notin H$ with $\vph'(y)\in\{\alpha,\beta\}$.
Now (a)
ensures the extension is a proper $L$-coloring; (b) ensures that an
$(\alpha,\beta)$-swap at $w$ will not create a problem at $x$; and (c)
ensures that each $(\alpha,\beta)$-component of $\vph$ contains at most one
$(\alpha,\beta)$-component of $\vph'$ and that an $(\alpha,\beta)$-swap at $w$
will not create a problem at $y$.

We form a list assignment $L'(x)$ from $L(x)$ by first removing each color that
is used by $\vph'$ on a
neighbor of $x$.  Further, if $\alpha$ is used on a neighbor of $x$ and
$\alpha\notin L(x)$, then we remove $\beta$ from $L(x)$.  Similarly, if $\alpha$
is used on two neighbors of $x$, then we remove $\beta$ from $L(x)$, regardless
of whether or not $\alpha\in L(x)$.  We also 
remove $\alpha$ from $L(x)$ if either of these situations occurs, but with
$\beta$ and $\alpha$ interchanged.  Since $|L(x)|\ge d_G(x)$, we must have
$|L'(x)|\ge d_H(x)\ge 1$, because $H$ is connected and $|V(H)|\ge 2$.  To extend
$\vph'$ to $x$, we simply choose any color in $L'(x)$.  This completes the key
step, started in the previous paragraph.  To extend $\vph'$ to $G-C$, we
repeatedly apply the key step, coloring vertices in order of non-increasing
distance from $C$.  Now we show how to extend $\vph'$ to $C$.  

\textbf{Case 1: $\bm{C}$ has no chord.}
For each $v\in V(C)$, form $L'(v)$ as in
the previous paragraph.  Again, $|L'(v)|\ge 2$ for all $v\in V(C)$.  First
suppose there exists $\gamma$ and $x,y\in V(C)$ such that
$\gamma\notin\{\alpha,\beta\}$ and $xy\in E(C)$ and $\gamma\in L'(x)\setminus
L'(y)$.  Now we color $x$ with $\gamma$ and proceed around $C$ finishing with
$y$.  
Each time that we color a vertex on $C$, we ensure that its degree in the
subgraph induced by vertices colored $\alpha$ and $\beta$ (among all colored
vertices, both inside and outside $C$) is at most 1.  This ensures that 
each $(\alpha,\beta)$-component of $\vph$ contains the vertex set of at most
one $(\alpha,\beta)$-component of $\vph'$.
This process succeeds because each time we color another vertex $z_1$ we reduce
the number of allowable colors on its uncolored neighbor $z_2$ by at most one (even if
we completely repeat the process of removing colors as in the previous
paragraph, treating $z_1$ as though it is outside $H$); this is virtually the
same as extending $\vph'$ to $G-C$.  We can finish at $y$
because the color on $x$ does not restrict our choice \mbox{of color for $y$}.

Suppose instead there exists $\gamma\notin\{\alpha,\beta\}$ such that
$\gamma\in L'(v)$ for all $v\in C$.  Now we use $\gamma$ on one maximum
independent set in $C$ and color each remaining vertex $v$ from
$L'(v)\setminus\{\gamma\}$.  

Finally, suppose that $L'(v)=\{\alpha,\beta\}$ for
each $v\in V(C)$.  Now we alternate $\alpha$ and $\beta$ around $C$.
In each case, it is easy to check that the resulting coloring $\vph$ is
$(\alpha,\beta)$-versatile at $w$. The key observation for this last case is
that no vertex of $C$ has a neighbor in $G-C$ colored $\alpha$ or $\beta$, by
construction of $L'(v)$.

\textbf{Case 2: $\bm{C}$ has a chord.}  Let $x$ denote one endpoint of the chord and
let $y$ and $z$ denote the neighbors of $x$ on $C$, besides the other endpoint
of the chord.  Form $L'(v)$ for each $v\in V(C)$, as above; again $|L'(v)|\ge
d_H(v)$ for all $v\in V(C)$.  Since $|L'(x)|\ge d_H(x)=3$, there exists
$\gamma\in L'(x)\setminus\{\alpha,\beta\}$.  If $\gamma\in L'(y)\cap L'(z)$,
then use $\gamma$ on $y$ and $z$ and color greedily toward $x$ in the remaining
uncolored subgraph.  So assume instead that $\gamma\notin L'(y)\cap L'(z)$; by
symmetry, assume that $\gamma\notin L'(z)$.  Now use $\gamma$ on $x$, then color
the remaining uncolored subgraph greedily in order of non-increasing distance from
$z$.  Again, we can finish at $z$ because using $\gamma$ on $x$ does not
restrict the choice of color for $z$.
As above,
each time that we color a vertex on $C$, we ensure that its degree in the
subgraph induced by vertices colored $\alpha$ and $\beta$ (%
both inside and outside $C$) is at most 1.  Again, this ensures that 
each $(\alpha,\beta)$-component of $\vph$ contains the vertex set of at most
one $(\alpha,\beta)$-component of $\vph'$.
\end{proof}

Now we prove Lemma~\ref{H-lem}.  As mentioned above, the proof mirrors that of
Lemma~\ref{degen-lem}, but now Lemma~\ref{versatile-lem} ensures some extension
to $H$ is versatile for the next Kempe swap in $G-H$.

\begin{proof}[Proof of Lemma~\ref{H-lem}.]
Fix a graph $G$ and a function $f:V(G)\to \Z^+$.  Let $H$ be an induced subgraph
of $G$ such that $G-H$ is $f$-swappable.  Let $f'(x):=f(x)-(d_G(x)-d_H(x))$ for
all $x\in V(H)$.  Assume that $f'(x)\ge d_H(x)$ for all $x\in V(H)$ and that
$H$ is $f'$-swappable.  

Let $L$ be an $f$-assignment for $G$.  Let $G':=G-H$.  By assumption,
each component of $G'$ is $L$-swappable.  So $G'$ is $L$-swappable.  Let
$\vph_0$ and $\vph$ be two $L$-colorings of $G$, and let $\vph_0'$ and $\vph'$
denote their restrictions to $G'$.  Since $G'$ is $L$-swappable, there exists a
sequence $\vph_0', \vph_1', \ldots, \vph_k'=\vph'$ of $L$-colorings of $G'$
such that every two successive $L$-colorings differ by a single $L$-valid Kempe
swap.  By induction on $k$, we extend each $\vph_i'$ to an $L$-coloring
$\vph_i$ of $G$ such that every two successive $L$-colorings in the sequence
$\vph_0,\vph_1,\ldots,\vph_k=\vph$ are $L$-equivalent. The case $k=0$ is easy
because $\vph_0=\vph$, so we are done.

So assume that $k\ge 1$.  
Suppose that $\vph_{i+1}'$ differs from $\vph_i'$ by an $\alpha,\beta$-swap at a
vertex $v_i$.  By Lemma~\ref{versatile-lem}, there exists an $L$-coloring
$\widetilde{\vph_i}$ of $G$, such that the restriction of $\widetilde{\vph_i}$
to $G'$ is $\vph'_i$ and an $\alpha,\beta$-swap at $v_i$ is $L$-valid in 
$\widetilde{\vph_i}$.  Furthermore, the restriction of this new coloring (after
performing the $\alpha,\beta$-swap at $v_i$) is $\vph_{i+1}'$.
It now suffices to show that $\vph_i$ and $\widetilde{\vph_i}$ are
$L$-equivalent. We do this by a sequence of Kempe swaps that recolors $H$ but
never changes the colors on $V(G-H)$.

For each $v\in V(G-H)$, remove $\vph'_{i}(v)$ from $L(w)$ for each $w\in
N(v)\cap V(H)$; denote the resulting list assignment on $H$ by $L_H$.  Note
that $|L_H(x)|\ge f'(x)\ge d_H(x)$ for all $x\in V(H)$.  If $|L_H(x)|>f'(x)$
for some $x$, then $H$ is $L_H$-swappable by Corollary~\ref{degen-cor}, since
$|L_H(y)|\ge f'(y)\ge d_H(y)$ for all $y\in H$.  So assume $|L_H(y)|=f'(y)$ for
all $y\in H$.  Since $H$ is $f'$-swappable, the restrictions of $\vph_i$ and
$\widetilde{\vph_i}$ to $H$ (which are both $L_H$-colorings) are
$L_H$-equivalent.  Consider a sequence of Kempe swaps that witnesses this. 
Note that performing the same Kempe swaps in $G$ transforms $\vph_i$ to
$\widetilde{\vph_i}$ (this is because for each edge $vw$ with $v\in V(H)$ and
$w\notin V(H)$, we have $\vph'_i(w)\notin L_H(v)$).  Now performing an
$\alpha,\beta$-swap at $v_i$ in $\widetilde{\vph_i}$ yields
an $L$-coloring of $G$ that restricts to $\vph'_{i+1}$ on $H$; denote
this $L$-coloring of $G$ by $\vph_{i+1}$.

The previous two paragraphs show that we can use $L$-valid Kempe swaps to
transform $\vph_0$ into an $L$-coloring $\widetilde{\vph}$ that agrees with
$\vph$ on $G-H$.  Finally, we transform $\widetilde{\vph}$ to $\vph$.
This is possible precisely because $H$ is $f'$-swappable.
\end{proof}

\section{Degree-swappable Graphs}
\label{degree-swap-sec} 

In this section we prove that various graphs are degree-swappable.  In view of
Lemma~\ref{H-lem}, if a connected $k$-regular graph $G$ contains an induced copy of any
degree-swappable graph, then $G$ is $k$-swappable.
Our first example of degree-swappable graphs requires a new definition.
A \Emph{theta graph}, $\Theta_{a,b,c}$, consists of two 3-vertices, $x$ and
$y$, that are linked by internally disjoint paths of lengths $a$, $b$, and $c$.

\begin{lem}
If $G$ is a bipartite theta graph, $\Theta_{a,b,c}$, then $G$ is degree-swappable.
\label{bipartite-lem}
\end{lem}
\begin{proof}
Fix a degree-assignment $L$ for $G$. Note that $G-vw$ is degree-choosable for
all $vw\in E(G)$. So, by Lemma \ref{overlap-lem}, $|L(v)\cap L(w)|\geq2$ for
all $vw\in E(G)$. Let $x$ and $y$ denote the 3-vertices of $G$, and let $P_a,P_b$,
and $P_c$ denote the internally disjoint $x,y$-paths of lengths $a$, $b$, and $c$, respectively. 
Since $G$ has an $x,y$-path with all internal vertices of degree 2, 
by Lemma~\ref{overlap-lem} we can assume,
by transitivity, that $|L(x)\cap L(y)|\ge 2$.

Suppose
$L(x)\neq L(y)$. Specifically, suppose $L(x)=\{1,2,3\}$ and $L(y)=\{1,2,4\}$. 
This implies that every 2-vertex has the list $\{1,2\}$. 
Fix an arbitrary $L$-coloring $\vph$ of $G$ with $\vph(x)=3$ and $\vph(y)=4$.
Starting from an arbitrary $L$-coloring $\vph'$, we can recolor $x$ with 3 and
recolor $y$ with 4, then reach $\vph$ by using at most one $1,2$-swap on each
path of $G-\{x,y\}$.
Thus, we assume that $L(x)=L(y)=\{1,2,3\}$. 

Suppose two disjoint $x,y$-paths, say $P_a$ and $P_b$, have the same list for their
2-vertices.  Form $G'$ from $G$ by deleting all 2-vertices of $P_a$.
Note that $d_{G'}(x)=2<3=|L(x)|$, so $G'$ is $L$-swappable by
Lemma~\ref{extra-lem}(a) with $w:=x$.  
This implies that $G$ is also $L$-swappable, as follows.
Since $P_a$ and $P_b$ have lengths of the same
parity, given any $L$-coloring of $G$, we can recolor the internal vertices
of $P_a$ so the neighbors of $x$ and $y$ on $P_a$ have the same colors as
their neighbors on $P_b$.  
Thus, any sequence of Kempe swaps in $G'$ extends to a sequence in $G$.
So we assume that no two $x,y$-paths have the same list for their 2-vertices. 

Denote by $x_a$, $x_b$, and $x_c$ the 2-neighbors of $x$, if they exist (when one
$x,y$-path has no 2-vertices, assume it is $P_c$, and let $x_c$ denote $y$).
Assume $L(x_a)=\{1,2\}$, $L(x_b)=\{1,3\}$, and $L(x_c)=\{2,3\}$.  Note that the
sets $\LL_{x,1}, \LL_{x,2}$, and $\LL_{x,3}$ each mix by
Lemma~\ref{missing-lem}.  

Consider $\vph\in\LL_{x_a,1}\cap\LL_{x_b,1}$.  Note that the 2,3-component
containing $x$ is a path (also containing all 2-vertices of $P_c$ and possibly $y$,
depending on the parities of the $x,y$-paths).  So a 2,3-swap at $x$ shows that
$\LL_{x,2}$ mixes with $\LL_{x,3}$.  By symmetry, $\LL_{x,1}$ mixes with
$\LL_{x,2}$.  Thus, $\LL_{x,1}\cup\LL_{x,2}\cup\LL_{x,3}$ mixes; that is, $\LL$
mixes.
\end{proof}

Recall that an even length cycle with at most one chord is a \emph{good cycle}.

\begin{lem}
Let $G$ be a graph that contains as induced subgraphs two good cycles, $H_1$ and
$H_2$.  Assume that $H_1$ and $H_2$ intersect in at most one vertex and that
all but at most one edge induced by $V(H_1)\cup V(H_2)$ either lies in $H_1$ or
lies in $H_2$; further, if there exists such an edge, then $V(H_1)\cap
V(H_2)=\emptyset$.  If $P$ is a shortest path from $H_1$ to $H_2$, then
$G[V(H_1)\cup V(H_2)\cup V(P)]$ is degree-swappable.  Thus, if $G$ is a
connected graph with at least two degree-choosable blocks, then $G$ is
degree-swappable.
\label{good-cycle-lem}
\end{lem}

\begin{figure}[!hb]
\centering
\begin{tikzpicture}[scale=-.4, yscale=1.2]
\tikzstyle{uStyle}=[shape = circle, minimum size = 5.5pt, inner sep = 0pt,
outer sep = 0pt, draw, fill=white]
\tikzset{every node/.style=uStyle}

\begin{scope}[xshift=1.4in]
\draw[thick] (0,0) node {} -- (1,1) node {} -- (2,0) node {} -- (1,-1) node
{} -- (0,0) (2,0) -- (3,1) node {} -- (4,0) node {} -- (3,-1) node {} -- (2,0)
node {} -- (4,0);
\end{scope}

\begin{scope}[xshift=-3in]
\draw[thick] (0,0) node {} -- (1,1) node {} -- (2,0) node {} -- (1,-1) node
{} -- (0,0) (2,0) -- (4,0) node {} -- (5,1) node {} -- (7,1) node {} -- (8,0)
node {} -- (7,-1) node {} -- (5,-1) node {} -- (4,0) node {};
\end{scope}

\begin{scope}[xshift=-7.35in]
\draw[thick] (-1,0) node {} -- (0,1) node {} -- (1,0) node {} -- (0,-1) node
{} -- (-1,0) -- (1,0) -- (2.5,0) node{} -- (4,0) node {} -- (5,1) node {} -- (7,1) node {} -- (8,0)
node {} -- (7,-1) node {} -- (5,-1) node {} -- (4,0) node {} (5,1) -- (5,-1);
\end{scope}

\end{tikzpicture}
\caption{Three examples of $G[V(H_1)\cup V(H_2)\cup V(P)]$ in
Lemma~\ref{good-cycle-lem}.}
\end{figure}
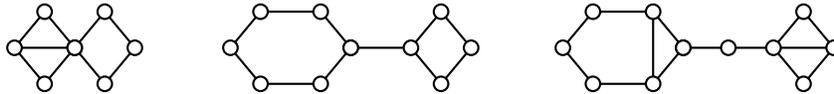

\begin{proof}
To begin, we prove the second statement from the first.
Let $B_1$ and $B_2$ be two degree-choosable blocks of $G$. 
So neither $B_1$ nor $B_2$ is a complete graph or odd cycle.
By Rubin's Block
Lemma~\cite{ERT},
there exist induced even cycles $H_1$ and $H_2$, each with at most one chord, in
$B_1$ and $B_2$, respectively. Let $P$ be a shortest path from $H_1$ to $H_2$. By
the first statement and the Key Lemma, $G$ is degree-swappable. 

Now we prove the first statement. If possible, choose $H_1$ so that it is an
even cycle with at most one chord that is closest to $H_2$. By
Lemma~\ref{H-lem}, it suffices to consider the case that $V(G)=V(H_1)\cup
V(H_2)\cup V(P)$. Let $L$ be a degree-assignment for $G$. Let $u$ be the
endpoint of $P$ in $H_1$.  For each path $Q$, let \Emph{$\ell(Q)$} denote the
length (number of edges) of $Q$.  We often write the \Emph{parity of $Q$} to
mean the parity of $\ell(Q)$.  If
$\ell(P)\neq0$, let $v$ be the neighbor of $u$ on $P$.  Suppose
$N_{H_1}(v)=\{w_1,\dots,w_t\}$ with $t\ge 3$. Let $P_i$ be the $w_i,w_{i+1}$-path
in $H_1$ for every $i\in[t]$. If $P_i$ is even for some $i\in[t]$, then
$G[v\cup V(P_i)]$ is an even cycle closer to $H_2$, contradicting our choice of
$H_1$. So assume $P_i$ is odd for every $i\in[t]$. Since $H_1$ is even,
$t$ must also be even; thus, $t\geq4$. But now $G[v\cup V(P_1)\cup
V(P_2)]$ is an even cycle closer to $H_2$. Thus, when $\ell(P)\neq0$, we assume
$|N_{H_1}(v)|\leq2$. 

Also, since $G-xy$ is degree-choosable for every $xy\in E(H_1)$, we assume
$|L(x)\cap L(y)|\geq2$ by Lemma~\ref{overlap-lem}.
Thus, if $w,z\in V(H_1)$ and some $w,z$-walk consists only of 2-vertices
(including $w$ and $z$), then transitivity implies $L(w)=L(z)$.

\textbf{Case 1: $\bm{\ell(P)=0}$ or $\bm{|N_{H_1}(v)|=1}$.} 

\textbf{Case 1.1: $\bm{H_1}$ has no chord.} Pick $u_1,u_2\in N_{H_1}(u)$ and
let $L(u_1)=\{a,b\}$; see Figure~\ref{no-chord}. As observed above,
$L(u_2)=L(u_1)=L(u_i)$ for all $u_i\in V(H_1)\setminus \{u\}$, by transitivity.
Since $H_1$ is even,
$\vph(u_1)=\vph(u_2)$ for every $\vph\in\LL$. By Corollary~\ref{degen-cor},
$\LL_{u_1,a}$ and $\LL_{u_1,b}$ each mix. 
To see this, order the vertices of $G-H_1$ by non-increasing distance from $u$
(with $u$ last), and recall that always $\vph(u_1)=\vph(u_2)$.
Now we show that $\LL_{u_1,a}\cup\LL_{u_1,b}$ mixes.
Since $H_2$ is degree-choosable and $|L(u)|=d(u)=3$, there exists an
$L$-coloring $\vph'$ with $\vph'(u)\notin\{a,b\}$.
Performing an $a,b$-swap at $u_1$ in $\vph'$ shows that $\LL_{u_1,a}$ mixes with
$\LL_{u_1,b}$.  Thus, $\LL$ mixes.

\textbf{Case 1.2: $\bm{H_1}$ has a chord $\bm{xy}$ with $\bm{u\notin\{x,y\}}$.}
Let $P_1$, $P_2$, and $P_3$ be the $x,y$-path, $x,u$-path, and $y,u$-path on $H_1$
avoiding $u$, $y$, and $x$, respectively; see Figure \ref{chord-xy}. If $\ell(P_1)$ is
odd, then $\ell(P_2\cup P_3)$ is also odd since $H_1$ is even. Thus, $H_1$
induces a bipartite theta graph, and we are done by Lemma~\ref{bipartite-lem}.
So we instead assume $\ell(P_1)$ is even. 

\begin{figure}[b!] 
\begin{center}
\begin{subfigure}{0.3\textwidth}
\centering
\begin{tikzpicture}[every node/.style={scale=0.8}]
\draw[thick] (1,0) -- (0,1) (0,-1) -- (1,0);
\draw[thick] (0,1) edge[bend right=90, looseness=2, decorate, decoration={snake, amplitude=0.4mm}] (0,-1);
\draw[thick] (0,1) node[uStyle] {$u_1$};
\draw[thick] (0,-1) node[uStyle] {$u_2$};
\draw[thick] (1,0) node[uStyle] {$u$};
\end{tikzpicture}
\caption{\label{no-chord}}
\end{subfigure}%
\begin{subfigure}{0.3\textwidth}
\centering
\begin{tikzpicture}[every node/.style={scale=0.8}]
\draw[thick] (-0.5,1) edge[bend left=50, decorate, decoration={snake, amplitude=0.4mm}] (1,0) (-0.5,-1) edge[bend right=50, decorate, decoration={snake, amplitude=0.4mm}] (1,0) (-0.5,1) edge[bend right=90, looseness=1.3, decorate, decoration={snake, amplitude=0.4mm}] (-0.5,-1) (-0.5,1) -- (-0.5,-1);
\draw[thick] (-1,0) node[draw=white, fill=white] {$P_1$};
\draw[thick] (0.2,0.6) node[draw=white, fill=white] {$P_2$};
\draw[thick] (0.2,-0.6) node[draw=white, fill=white] {$P_3$};
\draw[thick] (-0.5,1) node[uStyle] {$x$};
\draw[thick] (-0.5,-1) node[uStyle] {$y$};
\draw[thick] (1,0) node[uStyle] {$u$};
\end{tikzpicture}
\caption{\label{chord-xy}}
\end{subfigure}%
\begin{subfigure}{0.3\textwidth}
\centering
\begin{tikzpicture}[every node/.style={scale=0.8}]
\draw[thick] (-1,0) edge[bend right=90, looseness=1.7, decorate, decoration={snake, amplitude=0.4mm}] (1,0) (-1,0) edge[bend left=90, looseness=1.7, decorate, decoration={snake, amplitude=0.4mm}] (1,0) (-1,0) -- (1,0);
\draw[thick] (0,0.75) node[draw=none, fill=none] {$P_1$};
\draw[thick] (0,-0.75) node[draw=none, fill=none] {$P_2$};
\draw[thick] (-1,0) node[uStyle] {$x$};
\draw[thick] (1,0) node[uStyle] {$y$};
\end{tikzpicture}
\caption{\label{chord-uy}}
\end{subfigure}%
\caption{The 3 subcases in Case 1. (a) Case 1.1: $H_1$ has no chord. (b) Case 1.2: $H_1$ has a chord $xy$. (c) Case 1.3: $H_1$ has a chord $uy$ (i.e. $x=u$).\label{case1}}
\end{center}
\end{figure}
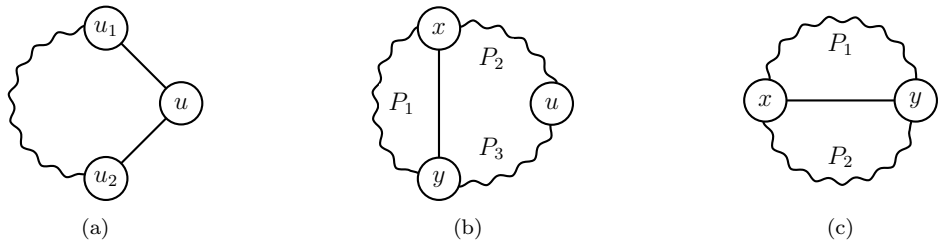

Assume $L(w)=\{a,b\}$ for every 2-vertex $w$
on $P_1$. Recall that $\{a,b\}\subseteq L(x)\cap L(y)$. Let $L(x)=\{a,b,c\}$.
Suppose $L(y)\neq L(x)$; so assume $L(y)=\{a,b,d\}$. By Lemma~\ref{missing-lem},
both $\LL_{x,c}$ and $\LL_{y,d}$ mix (and are nonempty). Further,
$\LL_{x,c}\cap\LL_{y,d}\neq\emptyset$, since $H_2$ is degree-choosable; thus,
$\LL_{x,c}\cup\LL_{y,d}$ mixes.
Since $G[V(P_1)]$ is an odd cycle, $\LL=\LL_{x,c}\cup\LL_{y,d}$, and we are
done. So we instead assume that $L(x)=L(y)=\{a,b,c\}$. 

Since $H_1$ and $\ell(P_1)$ are both even, $\ell(P_2)$ and $\ell(P_3)$ have the same parity. By
Lemma~\ref{missing-lem}, $\LL_{x,c}$ and $\LL_{y,c}$ each mix (and are nonempty). Further,
$\LL=\LL_{x,c}\cup\LL_{y,c}$. So it suffices to show that $\LL_{x,c}$ mixes
with $\LL_{y,c}$. Pick $x_1\in N_{P_2}(x)$ and $y_1\in N_{P_3}(y)$, if they exist;
again, see Figure~\ref{chord-xy}. 

Suppose first that neither $P_2$ nor $P_3$ has any internal vertex.  
Since $|L(u)\cap L(x)|\ge 2$, assume that $a\in L(u)$.  
Since $\ell(P_1)$ is even, we can
color $V(H_1)$ so that all
neighbors of $x$ and $y$ use color $a$.  Since $H_2$ is degree-choosable, we can
extend this coloring to an $L$-coloring $\vph$ of $G$.  Now performing a
$b,c$-swap at $x$ shows that $\LL_{x,c}$ mixes with $\LL_{y,c}$, and we are done.

Assume instead that $P_2$ or $P_3$ contains an internal 2-vertex.  By symmetry,
assume $x_1$ exists.  Let $w_1$ be a 2-vertex
adjacent to $x$ on $P_1$.  Recall that $L(w_1)=\{a,b\}$.  As above, we assume
$a\in L(x_1)$.  Now $\LL_{w_1,a}\cap \LL_{x_1,a}$ mixes by
Lemma~\ref{common-lem}.  To construct an $L$-coloring in this set, we can color
$x$ arbitrarily from $L(x)\setminus\{a\}$ and color greedily toward $H_2$.
Thus, $\LL_{w_1,a}\cap\LL_{x_1,a}$ contains $\vph_1$ and $\vph_2$ such that
$\vph_1(x)=c$ and such that $\vph_2(x)=b$, so $\vph_2(y)=2$.
This again proves that $\LL_{x,c}$ mixes with $\LL_{y,c}$, so we are done.

\textbf{Case 1.3: $\bm{H_1}$ has a chord $\bm{xy}$ with $\bm{x=u}$ (by
symmetry).} Let $P_1$ and $P_2$ be the $u,y$-paths in $H_1$; see Figure \ref{chord-uy}.
Since $H_1$ is even, $\ell(P_1)$ and $\ell(P_2)$ have the same parity. If
$\ell(P_1)$ is odd, then $H_1$ induces a bipartite theta graph, and we are done by Lemma
\ref{bipartite-lem}. So $\ell(P_1)$ is even. Pick $y_1\in N_{P_1}(y)$ and
$y_2\in N_{P_2}(y)$. By symmetry, assume $L(y)=\{a,b,c\}$ and $L(y_1)=\{a,b\}$.
By Lemma~\ref{missing-lem}, $\LL_{y,c}$ is nonempty and mixes. 
By Lemma~\ref{missing-lem}, $\LL_{y,c}$ mixes.

Suppose $L(y_2)=L(y_1)=\{a,b\}$. Pick $\vph\in\LL_{y,a}$. We show that $\vph$ is
$L$-equivalent to some $\vph'\in\LL_{y,c}$. If $\vph(x)\neq c$, then recoloring $y$ with $c$
gives a coloring in $\LL_{y,c}$. So assume $\vph(x)=c$. If $\vph(w)=b$ for some $w\in
N_{G-E(H_1)}(x)$, then there exists $\alpha\in L(x)$ which does not appear on
$N[x]$. We recolor $x$ with $\alpha$ then recolor $y$ with $c$ to get a
coloring in $\LL_{y,c}$. So assume $\vph(w)\neq b$ for every $w\in N_{G-E(H_1)}(x)$.
Now we perform an $a,b$-swap at $y$ followed by a $b,c$-swap at $y$ to get a
coloring in $\LL_{y,c}$.  The same argument shows that every
$\vph\in\LL_{y,b}$ is $L$-equivalent to some coloring in $\LL_{y,c}$. 
Since $\LL_{y,c}$ mixes, also
$\LL_{y,a}\cup\LL_{y,b}\cup\LL_{y,c}$ mixes; that is, $\LL$ mixes.

So we assume $L(x_1)\neq L(y_1)$. By symmetry, assume $L(y_1)=\{a,c\}$. By
Lemma~\ref{missing-lem}, both $\LL_{y,b}$ and $\LL_{y,c}$ mix (and are
nonempty). Further, $\LL_{y,a}$ is nonempty. And for every $\vph\in\LL_{y,a}$,
we can perform an $a,b$-swap (resp.~$a,c$-swap) at $y$ to get a coloring in
$\LL_{y,b}$ (resp.~$\LL_{y,c}$). Thus, $\LL_{y,a}\cup\LL_{y,b}\cup\LL_{y,c}$
mixes; that is, $\LL$ mixes.

\textbf{Case 2: $\bm{|N_{H_1}(v)|=2}$.}
Let $N_{H_1}(v)=\{v_1,v_2\}$; denote the $v_1,v_2$-paths in $H_1$ by $P_1$ and $P_2$. 
Recall, from the start of the proof, that $P_1$ and $P_2$ are both odd.

\textbf{Case 2.1: $\bm{H_1}$ has no chord.} By symmetry, assume $\ell(P_2)>1$.
Pick $w_1\in N_{P_2}(v_1)$ and $w_2\in N_{P_2}(v_2)$; see Figure \ref{2neighbors}. 
Assume
$L(v_1)=\{a,b,c\}$ and $L(w_1)=\{a,b\}$, by symmetry. 
Suppose $L(v_1)\neq
L(v_2)$; specifically, suppose $L(v_2)=\{a,b,d\}$. By Lemma~\ref{missing-lem},
both $\LL_{v_1,c}$ and $\LL_{v_2,d}$ mix. 
By Lemma~\ref{overlap-lem}, if $\ell(P_1)>1$, then $L(w)=\{a,b\}$ for every $w\in V(P_1)$. 
Thus, $\LL_{v_1,c}\cap\LL_{v_2,d}\neq\emptyset$. 
Further, for every $\vph\notin\LL_{v_1,c}\cup\LL_{v_2,d}$, there exists $\gamma\in(\{c,d\}-\vph(v))$.
So we can perform a Kempe swap either at $v_1$
or at $v_2$ to get a coloring in $\LL_{v_1,c}\cup\LL_{v_2,d}$.
Thus, $\LL$ mixes, and we are done. So we assume instead that $L(v_1)=L(v_2)=\{a,b,c\}$.

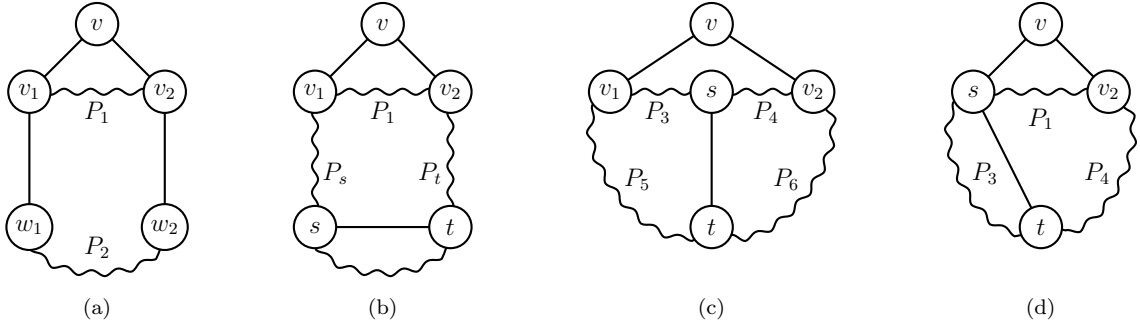
\begin{figure}[t!]
\begin{center}
\begin{subfigure}{0.23\textwidth}
\centering
\begin{tikzpicture}[every node/.style={scale=0.8}, scale=0.9]
\draw[thick] (-1,-1) -- (-1,1) -- (0,2) -- (1,1) -- (1,-1);
\draw[thick] (-1,-1) edge[bend right=90, decorate, decoration={snake, amplitude=0.4mm}] (1,-1) (-1,1) edge[decorate, decoration={snake, amplitude=0.4mm}] (1,1);
\draw[thick] (0,2) node[uStyle] {$v$};
\draw[thick] (-1,1) node[uStyle] {$v_1$};
\draw[thick] (1,1) node[uStyle] {$v_2$};
\draw[thick] (-1,-1) node[uStyle] {$w_1$};
\draw[thick] (1,-1) node[uStyle] {$w_2$};
\draw[thick] (0,0.7) node[draw=none, fill=none] {$P_1$};
\draw[thick] (0,-1.3) node[draw=none, fill=none] {$P_2$};
\end{tikzpicture}
\caption{\label{2neighbors}}
\end{subfigure}%
\begin{subfigure}{0.23\textwidth}
\centering
\begin{tikzpicture}[every node/.style={scale=0.8}, scale=0.9]
\draw[thick] (-1,1) -- (0,2) -- (1,1) (-1,-1) -- (1,-1);
\draw[thick] (-1,1) edge[decorate, decoration={snake, amplitude=0.4mm}] (1,1) (-1,1) edge[decorate, decoration={snake, amplitude=0.4mm}] (-1,-1) (1,1) edge[decorate, decoration={snake, amplitude=0.4mm}] (1,-1) (-1,-1) edge[bend right=90, decorate, decoration={snake, amplitude=0.4mm}] (1,-1);
\draw[thick] (0,2) node[uStyle] {$v$};
\draw[thick] (-1,1) node[uStyle] {$v_1$};
\draw[thick] (1,1) node[uStyle] {$v_2$};
\draw[thick] (-1,-1) node[uStyle] {$s$};
\draw[thick] (1,-1) node[uStyle] {$t$};
\draw[thick] (0,0.7) node[draw=none, fill=none] {$P_1$};
\draw[thick] (-0.7,-0.2) node[draw=none, fill=none] {$P_s$};
\draw[thick] (0.7,-0.2) node[draw=none, fill=none] {$P_t$};
\end{tikzpicture}
\caption{\label{st-p2}}
\end{subfigure}%
\begin{subfigure}{0.3\textwidth}
\centering
\begin{tikzpicture}[every node/.style={scale=0.8}, scale=0.9]
\draw[thick] (-1.5,1) -- (0,2) -- (1.5,1) (0,1) -- (0,-1);
\draw[thick] (-1.5,1) edge[decorate, decoration={snake, amplitude=0.4mm}] (0,1) (0,1) edge[decorate, decoration={snake, amplitude=0.4mm}] (1.5,1) (-1.5,1) edge[bend right=90, decorate, decoration={snake, amplitude=0.4mm}] (0,-1) (1.5,1) edge[bend left=90, decorate, decoration={snake, amplitude=0.4mm}] (0,-1);
\draw[thick] (-1,-1) edge[draw=white, bend right=90, decorate, decoration={snake, amplitude=0.4mm}] (1,-1); 
\draw[thick] (0,2) node[uStyle] {$v$};
\draw[thick] (-1.5,1) node[uStyle] {$v_1$};
\draw[thick] (1.5,1) node[uStyle] {$v_2$};
\draw[thick] (0,1) node[uStyle] {$s$};
\draw[thick] (0,-1) node[uStyle] {$t$};
\draw[thick] (-0.8,0.7) node[draw=none, fill=none] {$P_3$};
\draw[thick] (0.8,0.7) node[draw=none, fill=none] {$P_4$};
\draw[thick] (-1.1,-0.3) node[draw=none, fill=none] {$P_5$};
\draw[thick] (1.1,-0.3) node[draw=none, fill=none] {$P_6$};
\end{tikzpicture}
\caption{\label{st-p1-p2}}
\end{subfigure}%
\begin{subfigure}{0.23\textwidth}
\centering
\begin{tikzpicture}[every node/.style={scale=0.8}, scale=0.9]
\draw[thick] (1,1) -- (0,2) -- (-1,1) -- (0,-1);
\draw[thick] (-1,1) edge[decorate, decoration={snake, amplitude=0.4mm}] (1,1) (-1,1) edge[bend right=90, decorate, decoration={snake, amplitude=0.4mm}] (0,-1) (1,1) edge[bend left=90, decorate, decoration={snake, amplitude=0.4mm}] (0,-1);
\draw[thick] (-1,-1) edge[draw=white, bend right=90, decorate, decoration={snake, amplitude=0.4mm}] (1,-1); 
\draw[thick] (0,2) node[uStyle] {$v$};
\draw[thick] (-1,1) node[uStyle] {$s$};
\draw[thick] (0,-1) node[uStyle] {$t$};
\draw[thick] (1,1) node[uStyle] {$v_2$};
\draw[thick] (0,0.6) node[draw=none, fill=none] {$P_1$};
\draw[thick] (-0.83,-0.2) node[draw=none, fill=none] {$P_3$};
\draw[thick] (0.83,-0.2) node[draw=none, fill=none] {$P_4$};
\end{tikzpicture}
\caption{\label{st-x1-p2}}
\end{subfigure}
\end{center}
\caption{The 4 instances of Case 2, when $H_1$ contains 2 neighbors of $v$: the
first comprises Case~2.1 and the remaining three comprise Case~2.2. (a) $H_1$
has no chord. (b) $H_1$ has a chord $st$ on $P_2$. (c) $H_1$ has a chord $st$
on $P_1\cup P_2$. (d) $H_1$ has a chord $v_1t$ (i.e. $s=v_1$).}
\end{figure}

By Lemma~\ref{missing-lem}, $\LL_{v_1,c}$ and $\LL_{v_2,c}$ each mix. 
If $\ell(P_1)=1$ or $L(w)=\{a,b\}$ for every $w\in V(P_1)$, then $\LL_{v_1,a}$
and $\LL_{v_1,b}$ each mix by Lemma \ref{extra-lem}(a), with $w:=v_2$. 
This is because $\vph(w_2)=\vph(v_1)$ for every $\vph\in\LL_{v_1,a}\cup\LL_{v_1,b}$.
Similarly, $\LL_{v_2,a}$ and $\LL_{v_2,b}$ each mix. 
Further, $\LL_{v_1,a}\cap\LL_{v_2,b}\neq\emptyset$.  Also, for every
$\gamma\in\{a,b\}$, the set $\LL_{v_1,c} \cap \LL_{v_2,\gamma}\neq\emptyset$
and $\LL_{v_2,c} \cap \LL_{v_1,\gamma} \neq\emptyset$. 
So $\LL_{v_1,a}\cup\LL_{v_1,b}\cup\LL_{v_1,c}$ mixes; that is, $\LL$ mixes. 
Thus, we assume $\ell(P_1)>1$. Pick $y_1\in N_{P_1}(v_1)$ and $y_2\in N_{P_1}(v_2)$. 
By the above, we may assume $L(y_1)=L(y_2)=\{a,c\}$. 

By Lemma~\ref{missing-lem}, $\LL_{v_1,b}$ and $\LL_{v_2,b}$ each mix. Also,
$\LL_{v_1,c}\cap\LL_{v_2,b}\neq\emptyset$ and $\LL_{v_1,b}\cap\LL_{v_2,c}\neq\emptyset$.
We note that $\LL=\LL_{v_1,c}\cup\LL_{v_1,b}\cup\LL_{v_2,c}\cup\LL_{v_2,b}$.
So it suffices to show that $\LL_{v_2,c}\cup\LL_{v_2,b}$ mixes. As before,
$\LL_{v_1,a}$ mixes by Lemma \ref{extra-lem}(a), with $w:=v_2$. Moreover,
$\LL_{v_1,a}\cap\LL_{v_2,b}\neq\emptyset$ and
$\LL_{v_1,a}\cap\LL_{v_2,c}\neq\emptyset$. Thus,
$\LL_{v_2,c}\cup\LL_{v_2,b}$ mixes, and we are done.

\textbf{Case 2.2: $\bm{H_1}$ has a chord $\bm{st}$.} Recall, from the start of
the proof, that $H_1$ is an
even cycle (with at most one chord) closest to $H_2$; further, $\ell(P_1)$ and
$\ell(P_2)$ are odd. Assume $s,t\in V(P_2)-\{v_1,v_2\}$ (with $s$ closer to
$v_1$). Let $P_s$ (resp. $P_t$) be the $s,v_1$-path (resp. $t,v_2$-path)
avoiding $t$ (resp. avoiding $s$); see Figure \ref{st-p2}. If $P_s$ and $P_t$
have the same parity, then $G[H_1]$ is a bipartite theta graph (since $H_1$ is
even), and we are done by Lemma~\ref{bipartite-lem}. So assume $P_s$ and $P_t$
have opposite parities. Now $H_1[V(P_s)\cup V(P_t)\cup v]$ is an even cycle
closer to $H_2$, contradicting our assumption. The same argument works
(interchanging $P_1$ and $P_2$) if $s,t\in V(P_1)-\{v_1,v_2\}$. 

Assume instead that $s\in V(P_1)-\{v_1,v_2\}$ and $t\in V(P_2)-\{v_1,v_2\}$.
Let $P_3$, $P_4$, $P_5$, and $P_6$ be the $v_1,s$-path, $v_2,s$-path,
$v_1,t$-path, and $v_2,t$-paths forming $P_1$ and $P_2$; see Figure
\ref{st-p1-p2}. By symmetry, assume $P_3$ is even and $P_4$ is odd. Now $P_5$
is even and $P_6$ is odd; otherwise, $G[H_1]$ is a bipartite theta graph, and
we are done by Lemma~\ref{bipartite-lem}. But $H_1[V(P_4)\cup V(P_5)\cup v]$ is
an even cycle closer to $H_2$, contradicting our choice of $H_1$. 

So the chord $st$ must have an endpoint in $\{v_1,v_2\}$; say $s=v_1$. Note that
$t\neq v_2$; otherwise, $G[H_1]$ is a bipartite theta graph, and we are done by
Lemma~\ref{bipartite-lem}.  By symmetry, assume $t$ is on $P_2$. Let $P_3$ and
$P_4$ be the $s,t$ and $v_2,t$-paths forming $P_2$; see Figure \ref{st-x1-p2}.
Now $P_3$ is even and $P_4$ is odd; otherwise, $G[H_1]$ is a bipartite theta
graph, and we are done by Lemma~\ref{bipartite-lem}. But again $H_1[V(P_4)\cup v\cup
s]$ is an even cycle closer to $H_2$, contradicting our choice of $H_1$. 
\end{proof}

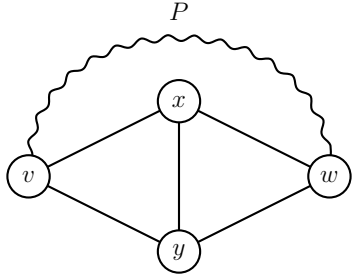
\begin{figure}[h!]
\begin{center}
\begin{tikzpicture}[every node/.style={scale=0.8}, rotate=90] 
\draw[thick] (1,0) -- (0,2) -- (-1,0) -- (0,-2) -- cycle (-1,0) -- (1,0);
\draw[thick] (0,2) edge[bend left=90, looseness=1.5, decorate, decoration={snake, amplitude=0.4mm}] (0,-2);
\draw[thick] (1,0) node[uStyle] {$x$};
\draw[thick] (0,2) node[uStyle] {$v$};
\draw[thick] (-1,0) node[uStyle] {$y$};
\draw[thick] (0,-2) node[uStyle] {$w$};
\draw[thick] (2.2,0) node[uStyle, draw=none, fill=none] {$P$};
\end{tikzpicture}
\end{center}
\caption{A $K_4^+$ formed by subdividing an edge one or more times, and the resulting path $P$.\label{K4plus}}
\end{figure}

\begin{lem}
The graph $K_4^{+}$ formed from $K_4$ by subdividing a single edge one or more times 
is degree-swappable.
\label{special graph}
\label{K4p-lem}
\label{K4p}
\end{lem}

\begin{proof}
Let $P$ be the
path formed by subdividing the edge in $K_4$ one or more times.
Denote the 3-vertices of $G$ by $v,w,x,y$, where $v$ and $w$ each have
a 2-neighbor; see Figure~\ref{K4plus}. Let $L$ be a degree-assignment for $G$.
By symmetry and by Lemma~\ref{overlap-lem}, we assume that $L(z)=\{1,2\}$ for
every $z\in V(P)$ and $\{1,2\}\subseteq L(v)\cap L(w)$. By symmetry, assume
that $L(v)=\{1,2,3\}$.

Assume first that $L(w)=L(v)=\{1,2,3\}$. By Lemma~\ref{missing-lem}, each of
$\LL_{v,3}$ and $\LL_{w,3}$ mix. Moreover,
$\LL_{v,3}\cap\LL_{w,3}\neq\emptyset$. Thus, $\LL_{v,3}\cup\LL_{w,3}$ mixes.
Let $\LL_1:=\LL_{v,3}\cup\LL_{w,3}$ and $\LL_2:=\LL\setminus\LL_1$. We now show
that every $\vph\in\LL_2$ mixes with $\LL_1$; thus $\LL$ mixes. Pick
$\vph\in\LL_2$ and assume by symmetry that $\vph(v)=1$. If
$3\notin\{\vph(x),\vph(y)\}$, then we recolor $v$ with 3 to get a coloring in
$\LL_1$, and we are done. So assume by symmetry that $\vph(x)=3$. Now suppose
$\vph(w)=1$. If $1\notin L(x)$, there exists $\gamma\in L(x)$ with
$\gamma\notin\cup_{z\in N[x]}\vph(z)$. So we recolor $x$ with $\gamma$ then
recolor $v$ with 3, and we are done. If, instead, $1\in L(x)$, then a 1,3-swap
at $v$ gives a coloring in $\LL_1$, and we are done. So assume instead that
$\vph(w)=2$. If $1\notin L(x)$, then $2\in L(x)$ by Lemma~\ref{overlap-lem}.
Now a 2,3-swap at $w$ gives a coloring in $\LL_1$, and we are done. If,
instead, $1\in L(x)$, then a 1,3-swap at $v$ gives a coloring in $\LL_1$, and
we are done.

Instead assume $L(w)\neq L(v)$; specifically, assume $L(w)=\{1,2,4\}$. By
Lemma~\ref{missing-lem}, the sets $\LL_{v,3}$ and $\LL_{w,4}$ each mix. Let
$\LL_1:=\LL_{v,3}\cup\LL_{w,4}$ and $\LL_2=:\LL\setminus\LL_1$. We show that
$\LL_1$ mixes. If $\{3,4\}\not\subseteq L(x)\cap L(y)$ or $L(x)\neq L(y)$, then
$\LL_{v,3}\cap\LL_{w,4}\neq\emptyset$; thus, $\LL_1$ mixes. Otherwise,
$L(x)=L(y)=\{3,4,\alpha\}$ with $\alpha\in\{1,2\}$ by Lemma~\ref{overlap-lem}.
By symmetry, assume $L(x)=L(y)=\{3,4,1\}$. We show again that $L_1$ mixes.

Suppose $\ell(P)$ is even. 
Pick $\vph\in\LL_{v,2}\cap\LL_{w,2}\cap\LL_{x,3}\cap\LL_{y,4}$. 
Now we can recolor either $x$ with 1 then $v$ with 3, or $y$ with 1 then $w$ with 4 to get colorings in $\LL_{v,3}$ and $\LL_{w,4}$, respectively. Thus, $\LL_1$ mixes. 
Instead assume that $\ell(P)$ is odd. Pick $\vph_1\in\LL_{v,1}\cap\LL_{w,2}\cap\LL_{x,3}\cap\LL_{y,4}$ and $\vph_2\in\LL_{v,2}\cap\LL_{w,1}\cap\LL_{x,3}\cap\LL_{y,4}$.
Now a 1,3-swap at $v$ in $\vph_1$ or a 1,4-swap at $w$ in $\vph_2$ give
colorings in $\LL_{v,3}$ and $\LL_{w,4}$, respectively. Also, $\vph_1$ and
$\vph_2$ are $L$-equivalent, as witnessed by a 1,2-swap at $v$. Thus, $\LL_1$ mixes.

Finally, we show that every $\vph^*\in\LL_2$ mixes with $\LL_1$; thus, $\LL$
mixes. Pick $\vph^*\in\LL_2$ with $\vph^*(v)=1$, by symmetry. Note that
$\{\vph^*(x),\vph^*(y)\}=\{3,4\}$; otherwise, we can either recolor $v$ with 3
or $w$ with 4 to get a coloring in $\LL_1$, and we are done. 
So assume $\vph^*(x)=3$ and $\vph^*(y)=4$. 
If $\ell(P)$ is odd, then we assume $\vph^*(v)=1$ and $\vph^*(w)=2$; if not,
then we achieve this with a 1,2-swap at $v$.  Now a 1,3-swap at $v$ shows that
$\vph^*$ mixes with $\LL_1$.  Instead assume that $\ell(P)$ is even.  Now assume
$\vph^*(v)=\vph^*(w)=2$; if not, then we achieve this by a 1,2-swap at $v$.
Recolor $x$ with 1, then recolor $v$ with 3.
\end{proof}

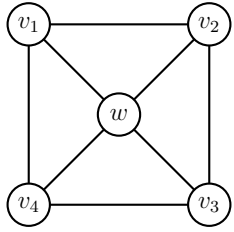
\begin{figure}[h!]
\begin{center}
\begin{tikzpicture}[every node/.style={scale=0.8}, xscale=-1, scale=1.2]
\draw[thick] (1,1) -- (-1,1) -- (-1,-1) -- (1,-1) -- cycle (1,1) -- (0,0) -- (-1,1) (-1,-1) -- (0,0) -- (1,-1);
\draw[thick] (1,1) node[uStyle] {$v_1$};
\draw[thick] (-1,1) node[uStyle] {$v_2$};
\draw[thick] (-1,-1) node[uStyle] {$v_3$};
\draw[thick] (1,-1) node[uStyle] {$v_4$};
\draw[thick] (0,0) node[uStyle] {$w$};
\end{tikzpicture}
\end{center}
\caption{The 4-wheel.\label{W4}}
\end{figure}

Let $W_4:=C_4\vee K_1$; this is the ``4-wheel'', or wheel with 4 spokes (see
Figure~\ref{W4}).

\begin{lem}
\label{W4-lem}
The graph $W_4$ is degree-swappable.
\end{lem}
\begin{proof}
Let $G:=W_4$.  Denote the 3-vertices by $v_1,v_2,v_3,v_4$, in order along a
4-cycle, and denote the dominating vertex by $w$; see Figure~\ref{W4}.  Fix a degree assignment $L$
for $G$.  Note that $|L(x)\cap L(y)|\ge 2$ for all $xy\in E(G)$, by
Lemma~\ref{overlap-lem}.  Since $|L(v_1)\cap L(v_2)|\ge 2$ and $|L(v_2)\cap
L(v_3)|\ge 2$ and $|L(v_2)|=3$, we conclude from Pigeonhole that $|L(v_1)\cap
L(v_3)|\ge 1$.  We consider the three cases $|L(v_1)\cap L(v_3)|\in\{1,2,3\}$.

\textbf{Case 1: $\bm{|L(v_1)\cap L(v_3)|=3}$.}
Assume $L(v_1)=L(v_3)$. 
If $L(v_2)\neq L(v_1)$, then there exists $\beta\in L(v_1)\setminus L(v_2)$
since $|L(v_2)|=|L(v_1)|$. Further, since $L(v_1)=L(v_3)$, we have $L(v_1)\cap
L(v_3)=L(v_1)$ so $\cup_{\alpha\in L(v_1)\cap L(v_3)}\LL_{v_1,\alpha}=\LL$. By
Lemma~\ref{common-lem}(3), $\LL$ mixes.
So assume $L(v_2)=L(v_1)$ and, by symmetry, $L(v_4)=L(v_1)$.  So there exists
$\alpha\in L(w)\setminus\bigcup_{i=1}^4L(v_i)$, and clearly $\LL_{w,\alpha}$
mixes.  Given an $L$-coloring $\vph$ with $\vph(w)\ne \alpha$, we can simply
recolor $w$ with $\alpha$.  Thus, $\LL$ mixes.

\textbf{Case 2: $\bm{|L(v_1)\cap L(v_3)|=2}$.}
Assume that $L(v_1)=\{a,b,c\}$ and $L(v_3)=\{a,b,d\}$.  
If $\{c,d\}\not\subseteq L(v_2)\cap L(v_4)\cap L(w)$, then $\LL$ mixes by
Lemma~\ref{common-lem}(3), as above.  
So assume $\{c,d\}\subseteq L(v_2)\cap L(v_4)\cap L(w)$.
By Case 1 and symmetry, we assume $L(v_2)\ne L(v_4)$.
Interchanging $v_2$ and $v_4$ with $v_1$ and $v_3$ (and
interchanging $c$ and $d$ with $a$ and $b$) shows that also $a,b\in L(w)$.
Further, $L(v_2)\cup L(v_4)\subseteq L(w)$.  Thus, 
$L(v_2)=\{a,c,d\}$ and $L(v_4)=\{b,c,d\}$,
up to possibly swapping $v_2$ and $v_4$. 
Now by Lemma~\ref{common-lem}(3), both
$\bigcup_{\alpha\in\{a,b\}}\LL_{v_1,\alpha}\cup\LL_{v_3,\alpha}$
and $\bigcup_{\alpha\in\{c,d\}}\LL_{v_2,\alpha}\cup\LL_{v_4,\alpha}$ mix.
The union of these two sets is all of $\LL$.  Further, the two sets mix, since
some coloring $\vph$ lies in both; namely, $\vph(v_1)=\vph(v_3)=a$,
$\vph(v_2)=\vph(v_4)=c$, and $\vph(w)=d$.  Thus, $\LL$ mixes.

\textbf{Case 3: $\bm{|L(v_1)\cap L(v_3)|=1}$.}
Assume $L(v_1)=\{a,b,c\}$ and $L(v_3)=\{a,d,e\}$.  By symmetry between $v_1,v_3$
and $v_2,v_4$, we assume that $|L(v_2)\cap L(v_4)|=1$.  Since $|L(v_i)\cap
L(v_j)|\ge 2$ for all $i\in\{1,3\}$ and $j\in\{2,4\}$, we assume that
$L(v_2)=\{a,b,d\}$ and $L(v_4)=\{a,c,e\}$.  
Since $c\in L(v_1)\setminus L(v_2)$, Lemma~\ref{common-lem}(3) shows that
$\LL_{v_1,a}\cup \LL_{v_1,c}$ mixes.
Since $b\in L(v_1)\setminus L(v_4)$, Lemma~\ref{common-lem}(3) shows that
$\LL_{v_1,a}\cup \LL_{v_1,b}$ mixes.  Thus, $\LL$ mixes.
\end{proof}

\section{Proof of Main Theorem}
\label{main-proof-sec} 

In this section, we prove our main result: Every connected $k$-regular graph
(with $k\ge 3$) is $k$-swappable unless $G=K_{k+1}$ or $G=K_2\,\square\, K_3$.
Further, these exceptional graphs are $L$-swappable whenever $L$ is a
$\Delta$-assignment that is not identical everywhere.
We split the main result into 4 cases depending on the connectivity of $G$. In
Theorem~\ref{connectivity1}, we prove the result for connectivity 1. In
Theorem~\ref{connectivity2}, we prove the result for connectivity 2 by using a
cut set of size 2 (a \textit{2-cut}), and showing that $G$ contains an induced
degree-swappable subgraph from the family of graphs compiled in 
Section~\ref{degree-swap-sec}. 

For connectivity 3, we split into three cases: (i) $k\geq5$, which we prove in
Lemma~\ref{connectivity3-main-lem}, (ii) $k=4$, which we prove in
Lemma~\ref{connectivity3-4reg-lem}, and (iii) $k=3$, which we prove in
Lemma~\ref{connectivity3-3reg-lem}. For cases (i) and (ii) we use a 3-cut and
show that $G$ contains an induced degree-swappable subgraph from 
Section~\ref{degree-swap-sec}.  For
case (iii), we first show that for every 3-assignment $L$ of $G$, the lists
must be identical for all vertices. Then, in Theorem~\ref{connectivity3}, we
invoke the result of Feghali, Johnson, and Paulusma for 3-colorings of
3-regular graphs \cite[Theorem 1]{FJP} (for completeness, we state it below).
Finally, we prove the result for 4-connected graphs in
Theorem~\ref{connectivity4} as follows. In Lemmas~\ref{4connected-lem}
and~\ref{clique-lem}, we handle the case of
$k$-assignments $L$ that are not identical for all
vertices (Lemma~\ref{clique-lem} handles the case $G=K_{k+1}$). In
Lemma~\ref{noW4-coloring-lem}, we show that either $G$ contains a
4-wheel (so we are done by Lemma~\ref{W4-lem}), or the absence of a 4-wheel
restricts the possible $L$-colorings enough that $G$ must be $L$-swappable.

\begin{thmA}{\cite[Theorem 1]{FJP}}
If $G$ is a connected 3-regular graph that is neither $K_4$ nor the graph
$K_2\,\square\, K_3$, then all 3-colorings of $G$ are 3-equivalent.
\label{FJP Theorem}
\end{thmA}

\begin{thm}
For $k\ge 3$, if $G$ is $k$-regular with connectivity 1, then $G$ is $k$-swappable.
\label{connectivity1}
\end{thm}

\begin{proof}
Since $G$ has connectivity 1, it contains a cut-vertex. Thus, $G$ contains at
least two endblocks. If $G$ contains at most one degree-choosable block, then
some endblock $B$ is not degree-choosable. Let $v$ be the cut-vertex in $B$.
Since $k\geq3$ and $G$ is $k$-regular, $B=K_{k+1}$ and $d_B(v)=k$. But now
$G=B$ which contradicts that $G$ has connectivity 1. Thus, $G$ contains at least
two degree-choosable endblocks. By Lemma~\ref{good-cycle-lem}, $G$ is $k$-swappable. 
\end{proof}

Now we consider the case that $G$ has a vertex cut of size at most 3.
In a graph $G$, a block is \EmphE{Gallai}{-4mm} if it an odd cycle or a clique;
otherwise it is \Emph{non-Gallai}.

\begin{thm}
For $k\geq3$, if $G$ is $k$-regular with connectivity 2, then $G$ is $k$-swappable.  
\label{connectivity2}
\end{thm}

\begin{proof}
Let $G$ satisfy the hypotheses of the theorem, and let $L$ be a $k$-assignment
for $G$. Let $S$ be a 2-cut in $G$.  Let $G_S:=G-S$\aside{$G_S$}.
Now every endblock $B$ of $G_S$ has order at least $k-1$, since each
vertex $v$ of $B$ has $d_G(v)=k$ and $v$ has at most two edges to $S$.
Further, if an endblock $B$ is non-Gallai, then $B$ contains an even cycle with at most one
chord.\footnote{Recall that Rubin's Block Lemma~\cite{ERT} says 
a block contains such a cycle if and only if it is non-Gallai.}  
We call such a cycle a \Emph{good cycle}. 
If $G_S$ has at least two endblocks each of which contains a {good
cycle}, then Lemma~\ref{good-cycle-lem} implies that $G$ is $L$-swappable. 
So we assume that at most one endblock of $G_S$ contains a good cycle.
 Thus, at most one endblock of $G_S$ is non-Gallai.

\begin{clm}$G_S$ has at most two Gallai endblocks.
\end{clm}
\begin{clmproof}
Observe that every Gallai endblock $B$ is regular of degree either $k-1$ or $k-2$,
since $|S|=2$.  Further, every such $B$ sends at least $k-1$ edges to $S$.
This holds because either (i) $B$ is regular of
degree $k-1$, so $B$ has at least $k$ vertices, and at least $k-1$ of these
each send an edge to $S$, or (ii) $B$ is regular of degree $k-2$, so $B$ has
at least $k-1$ vertices, and at least $k-2$ of these each sends two edges to $S$. 

Suppose, contrary to the claim, that $S$ contains at least 3 Gallai endblocks,
$B_1$, $B_2$, $B_3$.
As noted above, each $B_i$ sends $S$ at least
$k-1$ edges.  So
$3(k-1)\le 2k$; thus, $k=3$.  Further, $G_S$ has no other
endblocks.  But now some $B_i$ is its own component, so it contains no
cut-vertex, and thus sends more edges to $S$ than counted above, which gives a
contradiction.  
\end{clmproof}

\begin{clm}$G_S$ has exactly one non-Gallai endblock; call it $B_0$.
\end{clm}
\begin{clmproof}
Suppose $G_S$ has no non-Gallai endblocks. Since $G_S$ has at most two (Gallai)
endblocks, each endblock is its own component.  So each endblock sends $S$ at least
$\min\{2\cdot(k-1),1\cdot k\}=k$ edges, for a total of at least $2k$ edges to $S$. 
If either component of $G_S$ is $(k-2)$-regular, then $S$ has too many incident
edges, so we get a contradiction.  Thus, each endblock is $K_k$.  If either
vertex in $S$ has at least 2 neighbors in each component of $G_S$, then
we are done by Lemma~\ref{good-cycle-lem}.  So assume that each vertex of
$S$ sends one edge to one component and sends $k-1$ edges to the other.
Now $G$ contains a copy of $K_4^+$, and we are done by Lemma~\ref{K4p-lem}.
Thus, instead $G_S$ has exactly one non-Gallai endblock; call it $B_0$.
\end{clmproof}
\smallskip

\textbf{Case 1: $\bm{G_S}$ has exactly one Gallai endblock; call it $\bm{B_1}$.} If
$B_1$ is $(k-1)$-regular, then some $v\in S$ has at least two neighbors in
$B_1$, but does not dominate $B_1$.  Thus, $B_1+v$ contains a good cycle, and
so does $B_0$.  Now we are done by Lemma~\ref{good-cycle-lem}.  So assume
instead that $B_1$ is $(k-2)$-regular and that $B_1$ has order
at least $k-1$.  In fact, each vertex of $S$ is adjacent to all of $B_1$.  
Now counting edges shows that $B_1=K_{k-1}$, each vertex of $S$ is adjacent to
all of $B_1$, and $S$ is an independent set.  Hence, $G$ contains a copy of
$K_4^+$, and we are done by Lemma~\ref{K4p-lem}.

\textbf{Case 2: $\bm{G_S}$ has exactly 2 Gallai endblocks, $\bm{B_1}$ and $\bm{B_2}$.} 
Suppose that $B_1$ or $B_2$ is its own component; by symmetry, say that it is
$B_1$.  If $B_1$ is $(k-1)$-regular, then some vertex $x\in S$ has at least two
neighbors in $B_1$, but does not dominate it.  So $B_1+x$ and $B_0$ each contain
good cycles, and we are done by Lemma~\ref{good-cycle-lem}.  
If $B_1$ is $(k-2)$-regular, then it sends at least $2(k-1)$ edges to $S$.
So the total number of edges incident to $S$ is at least
$2(k-1)+(k-1)+1=3k-2>2k$, a contradiction.  

Thus, $B_1$ and $B_2$ are in the same component $G_2$ of $G_S$ and 
$B_0$ is its own component $G_1$.  Now $S$ sends $2(k-1)$ edges to
non-cut-vertices in endblocks of $G_2$.  This implies, by counting edges to
$S$, that every other vertex of $G_2$ is non-adjacent to every vertex in $S$.  

If some vertex $x\in S$ sends at least two edges to $B_1$, then (since
$x$ sends no edges to the cut-vertex in $B_1$) $G_S$ has a good cycle in $B_0$
and another good cycle in $B_1+x$, so we are done by
Lemma~\ref{good-cycle-lem}.  The same is true if some $x\in S$ sends at least
two edges to $B_2$.  So assume that each $x\in S$ sends at most one edge to
$B_1$ and one edge to $B_2$.  Since $B_1$ and $B_2$ each receive at least $k-1$
edges from $S$, this implies that $k=3$ and each $x\in S$ sends exactly one
edge to each of $B_1$ and $B_2$.  Now $G_2+x$ must contain a good cycle, unless
$B_1=B_2=K_2$ and $G_2+x$ is an odd cycle. But in this case, each vertex of
$G_2$ that is not a non-cut-vertex of $B_1$ or $B_2$
has too few incident edges, a contradiction.
\end{proof}

\begin{lem}
For $k\geq 5$, if $G$ is $k$-regular with connectivity 3, then $G$ is $k$-swappable.  
\label{connectivity3-main-lem}
\end{lem}

\begin{proof}
Let $G$ satisfy the hypotheses of the lemma, and 
let $L$ be a $k$-assignment for $G$. Let $S$ be a minimal 3-cut of
$G$, and let $G_S:=G-S$\aside{$G_S$}. 

\setcounter{clm}{0}
\begin{clm}
\label{clm3}
Every Gallai endblock $B$ of $G_S$ is regular of degree $k-1$, $k-2$, or $k-3$
and sends at least $k-1$ edges to $S$.  If $B$ is a component of $G_S$,
then $B$ sends at least $k$ edges to $S$.
\end{clm}

\begin{clmproof}
Since $G$ is $k$-regular and $|S|=3$, every non-cut-vertex in an endblock $B$
sends at most 3 edges to $S$, so sends at least $k-3$ edges to vertices of $B$.
Since $B$ is a Gallai endblock, $B$ is
regular.  And since $G$ is 3-connected, some non-cut-vertex of $B$ 
sends an edge to $S$.  
Thus, $B$ is regular of degree $k-1$, $k-2$, or $k-3$.  

Note that every Gallai endblock $B$ of
$G_S$ sends at least $k-1$ edges to $S$. This
is because either (i) $B$ is $(k-1)$-regular, so at least $k-1$ vertices of $B$
each send 1 edge to $S$, or (ii) $B$ is $(k-2)$-regular, so at least
$k-2$ vertices of $B$ each send 2 edges to $S$, or (iii) $B$ is
$(k-3)$-regular, so at least $(k-3)$ vertices of $B$ each send 
3 edges to $S$. Thus, the number of edges that $B$ sends to $S$ is at
least $\min\{1\cdot(k-1),2\cdot(k-2),3\cdot(k-3)\}=k-1$.  Further, if a component of
$G_S$ consists of a single block, say $B$, then a similar computation
shows that the number of edges $B$ sends to $S$ is at least $k$.
\end{clmproof}

\begin{clm}
\label{clm4}
$G_S$ has no endblock $B$ that is an odd cycle with length at least 5.
\end{clm}
\begin{clmproof}
If such a $B$ exists, then it sends at least $4(k-2)$ edges to $S$.  
By Claim~\ref{clm3},  
each component not containing $B$ sends at least $k$ edges to $S$. 
But $4\cdot(k-2)+k=5k-8>3k$, a contradiction.
\end{clmproof}

\begin{clm}
\label{clm5}
$G_S$ has at most 1 non-Gallai endblock and at most 3 Gallai endblocks.
\end{clm}

\begin{clmproof}
If $G_S$ contains at least 2 non-Gallai endblocks, then it contains 2 {good
cycles}; so $G$ is degree-swappable, by Lemma~\ref{good-cycle-lem}.  So instead
$G_S$ contains at most one non-Gallai endblock.  
If $G_S$ has at least 4 Gallai endblocks, then by Claim~\ref{clm3} the number
of edges from $S$ to endblocks of $G_S$ is at least $4(k-1)>3k$, a
contradiction.  So $G_S$ has at most 3 Gallai endblocks.  
\end{clmproof}
\smallskip

We will show $G$ contains an induced degree-swappable subgraph, so we are
done by Lemma~\ref{H-lem}.  Specifically, we will often show $G$
contains two good cycles, and invoke Lemma~\ref{good-cycle-lem}.
\smallskip

\textbf{Case 1: $\bm{G_S}$ has 0 non-Gallai endblocks.}
Let $B_1$ be an endblock that is a whole component $G_1$.
 If $B_1=K_{k-2}$, then each vertex in $S$ is adjacent to all of $B_1$ (and
some pair of vertices in $S$ is non-adjacent), so $G$ contains an induced
$K_4^+$; thus, we are done by Lemma~\ref{K4p}.  So instead
$B_1\in\{K_{k-1},K_k\}$.  Suppose $B_1=K_k$.  Now some vertex $x\in S$ is
adjacent to at least two vertices of $B_1$ and also non-adjacent to some vertex
$y\in B_1$; see Figure~\ref{k5-3conn-noGB}.  (This follows from Pigeonhole and
the fact that each vertex of $S$ has a neighbor in $B_1$.) Choose $z\in
N(y)\cap S$.  Since $x$ and $z$ each have a neighbor in $G_2$, the subgraph
$G[S\cup V(G_2)]$ contains an $x,z$-path.  Note that $N(x)\cap N(z)\cap
V(B_1)=\emptyset$, since each vertex of $B_1$ has a unique neighbor in $S$.
Thus, $G$ contains a copy of $K_4^+$ (possibly using the edge $xz$, rather than
an $x,z$-path through $G_2$), and we are done by Lemma~\ref{K4p}.

Assume instead that $B_1=K_{k-1}$. Now each vertex of $B_1$ has exactly two
neighbors in $S$.  Denote $S$ by $\{x_1,x_2,x_3\}$.  By symmetry, there exist
vertices $y_1,y_2,y_3\in B_1$ such that $y_1$ and $y_2$ are both adjacent to
each of $x_1$ and $x_2$; and $y_3$ is adjacent to $x_2$ and $x_3$.  Now
$\{x_1,y_1,y_2,y_3\}$ induces $K_4-e$.  Since $x_3$ is adjacent to $y_3$ (but not
$y_1$ or $y_2$), again $G$ contains $K_4^+$ as above, so we are done by
Lemma~\ref{K4p}.
\smallskip

\textbf{Case 2: $\bm{G_S}$ has exactly 1 non-Gallai endblock, $\bm{B_1}$.}
Let $G_1$ be the component of $G_S$ containing $B_1$, and let $B_2$ be
a regular endblock of another component $G_2$. 
Since $S$ sends at least $k-1$ edges to $B_2$, by Pigeonhole some vertex 
$x\in S$ sends at least 2 edges to $B_2$. 
If $B_2$ has at least $k$ vertices, then $x$ is not adjacent to all of $B_2$,
since $x$ sends an edge to $G_1$.  Thus, $B_2+x$ is non-Gallai, so it contains
a {good cycle}.   Since $B_2+x$ and $B_1$ each contain a
good cycle, we are done by Lemma~\ref{good-cycle-lem}.

Assume instead that $B_2$ has order $k-1$ or $k-2$.   We will show that $G_2=B_2$. 
Assume the contrary; so $B_2$ has a cut-vertex $y$; see Figure~\ref{k5-3conn-GB}.  
If $B_2$ has order $k-2$, then $d_{G_2}(y)\ge d_{B_2}(y)+1=(|B_2|-1)+1=k-2$; so
$d_{G[S\cup\{y\}]}(y)= k-d_{G_2}(y)\le 2$ 
However, each other vertex of $B_2$ is adjacent to
all of $S$.  So some $x\in S$ is not adjacent to $y$ but sends at least two
edges to $B_2$. Thus, $B_2+x$ 
and $B_1$ each contain a good cycle, so we are
done by Lemma~\ref{good-cycle-lem}.

Assume instead that $B_2$ has order $k-1$; again, see Figure~\ref{k5-3conn-GB}.
The number of edges from $S$ to $B_2$ is at least 
$2(k-2)$,
and $d_{G_2}(y)\ge 1+d_{B_2}(y)=1+((k-1)-1)=k-1$, so
$d_{G[S\cup\{y\}]}(y)\le 1$.
Thus, $S$ sends $y$ at most 1 edge and in total sends all other vertices of
$B_2$ at least $2(k-2)$ edges. If two vertices of $S$ each send only one edge to $B_2$, 
then some non-cut-vertex of $B_2$ has too few incident edges. This is because
each non-cut-vertex of $B_2$ needs an edge from at least 2 vertices of $S$, and
$B_2$ has at least $(k-2)\geq3$ non-cut-vertices.  So at least 2 vertices of
$S$ send at least 2 edges to $B_2$, and at least one of them is non-adjacent to $y$.
We let $x$ denote such a vertex. Now $B_1$ and $B_2+x$ each contain a good cycle,
and we are done by Lemma~\ref{good-cycle-lem}. Thus, we assume that $G_2=B_2$.

\begin{figure}[t!]
\begin{center}
\begin{subfigure}{0.4\textwidth}
\centering
\begin{tikzpicture}[xscale=-1,rotate=90]
\draw[draw=none] (0,-4.95) circle (.5mm); 
\draw[thick] (0,-2) ellipse (1.5cm and 0.7cm) (0,-4) ellipse (2cm and 0.8cm); 
\draw[thick] (0,0) ellipse (0.9cm and 0.9cm); 
\draw[thick] (0.5,-4) ellipse (0.5cm and 0.5cm); 
\draw[thick] (0.5,-4) node[uStyle, draw=white] {$B_2$};
\draw[thick, dotted] (-0.5,-2) -- (0,-2) -- (0.5,-2); 
\draw[thick] 
(-0.5,-2) node[fill=white, shape=circle, scale=0.5,draw] (A) {} 
(0,-2) node[fill=white, shape=circle, scale=0.5,draw] (B) {} 
(0.5,-2) node[fill=white, shape=circle, scale=0.5,draw] (C) {}; 
\draw[thick] (0.5,-2.3) node {$x$};
\draw[thick] (0,-2.3) node {$z$};
\draw[thick] (-0.3,0.2) node {$y$};
\draw[thick] 
(0.3,-0.2) node[fill=white, shape=circle, scale=0.5,draw] (D) {} 
(-0.3,-0.2) node[fill=white, shape=circle, scale=0.5,draw] (E) {} 
(0.5,0.3) node[fill=white, shape=circle, scale=0.5,draw] (F) {}; 
\draw[thick] (F) -- (C) -- (D) (B) -- (E) (D) -- (E) -- (F) -- (D);
\draw[thick] (1.5,0) node {$B_1$} (2,-2) node {$S$} (2.5,-4) node {$G_2$};
\end{tikzpicture}
\caption{\label{k5-3conn-noGB}}
\end{subfigure}
\begin{subfigure}{0.5\textwidth}
\centering
\begin{tikzpicture}[xscale=-1,rotate=-90]
\draw[thick] (0,0) ellipse (1.2cm and 0.9cm) (0,-2) ellipse (1.5cm and 0.7cm) (0,-4) ellipse (2cm and 0.8cm); 
\draw[thick] (-0.3,0) ellipse (0.5cm and 0.5cm); 
\draw[thick] (-0.5,-4) ellipse (0.5cm and 0.5cm); 
\draw[thick] (-0.5,-4) node[uStyle, draw=white] {$B_2$};
\draw[thick] (-0.3,0) node[uStyle, draw=white] {$B_1$};
\draw[thick] (0,-4) node[fill=white, shape=circle, scale=0.5, draw] (A) {}; 
\draw[thick] (0.1,-3.6) node {$y$}; 
\draw[thick] (A) edge[decorate, decoration={snake, amplitude=0.4mm}] (0.8,-4); 
\draw[thick] (-0.5,-2) node[fill=white, shape=circle, scale=0.5, draw] (B) {} (0,-2)
node[fill=white, shape=circle, scale=0.5, draw] (C) {} (0.5,-2) node[fill=white,
shape=circle, scale=0.5, draw] (D) {}; 
\draw[thick, dotted] (B) -- (C) -- (D); 
\draw[thick] (-1.8,0) node {$G_1$} (-2,-2) node {$S$} (-2.5,-4) node {$G_2$};
\end{tikzpicture}
\caption{\label{k5-3conn-GB}}
\end{subfigure}%
\end{center}
\caption{$S$ and the components of $G_S$ when 
(a) $G_1$ is a single Gallai endblock $B_1$, and 
(b) $G_1$ contains a non-Gallai endblock $B_1$ and $G_2$ has more than one endblock.\label{k5-3conn}}
\end{figure}
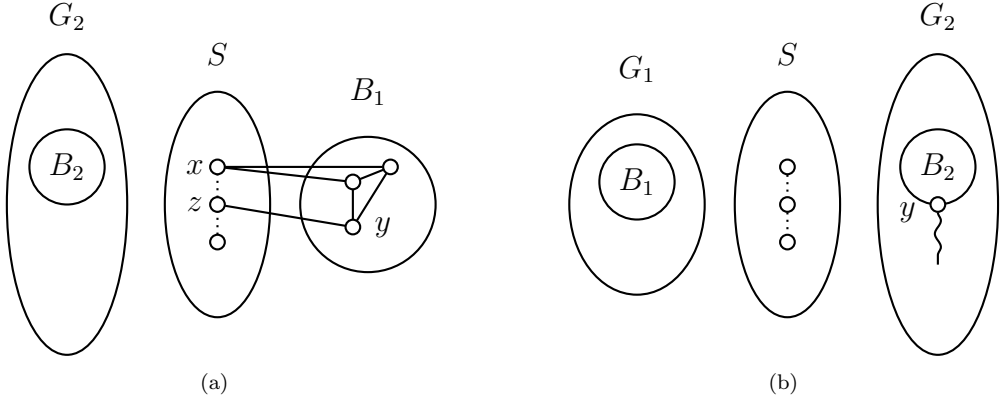

Suppose $B_2$ has $k-1$ vertices.  Now $S$ sends $B_2$ exactly $2(k-1)$
edges.  Thus, some vertex $x\in S$ sends at least two edges to $B_2$
and is not adjacent to all of $B_2$; this is because $(k-1)+2(1) < 2(k-1) <
2(k-1)+1$. (The first term represents the possibility of at
most one vertex of $S$ sending at least two edges to $B_1$ and the third term
represents the possibility of at least two vertices of $S$ dominating $B_1$.)
Now $B_2+x$ contains a {good cycle}, as does $B_1$, so we are done by
Lemma~\ref{good-cycle-lem}. Hence, $B_2$ has $k-2$ vertices, and every vertex
in $S$ is adjacent to all of $B_2$. At least one pair in $S$ is not adjacent;
call it $x_1,x_2$.  Each of $x_1$ and $x_2$ has a neighbor in $G_1$, so $G$
contains a copy of $K_4^+$, and we are done by Lemma~\ref{K4p-lem}.
%
%
\end{proof}

\begin{lem}
If $G$ is a $4$-regular graph with connectivity 3, then $G$ is $k$-swappable.  
\label{connectivity3-4reg-lem}
\end{lem}
\begin{proof}
Let $S$ be a vertex cut of size 3, and let $G_S:=G-S$\aside{$G_S$}.  As in the
proof of Lemma~\ref{connectivity3-main-lem}, if $G_S$ has at least two
non-Gallai endblocks, then we are done by Lemma~\ref{good-cycle-lem}.
So we assume that $G_S$ has at most one non-Gallai endblock.

\setcounter{clm}{0}
\begin{clm}
\label{clm6}
Each Gallai endblock of $G_S$ is $K_2$, $K_3$, or $K_4$.
\end{clm}
\begin{clmproof}
The proofs of Claims~\ref{clm3} and~\ref{clm4} in
Lemma~\ref{connectivity3-main-lem} also work here
essentially unchanged.  So each Gallai endblock is regular of degree $k-1$,
$k-2$, or $k-3$.  (Here $k=4$.)  As before, if $G_S$ has an endblock $B$ that
is an odd cycle with length at least 5, then the number of edges going to $S$
is at least $4(k-2)+k$.  However, we can strengthen this bound as follows.  If
$B$ is its own component, then the number of edges it sends to $S$ is at least $5(k-2)$.
 And if $B$ is not its own component, then some other endblock in its component
sends $S$ at least $k-1$ edges.  Thus, we get
$\min\{4\cdot(k-2)+(k-1),5(k-2)\}+k>3k$, again a contradiction.
\end{clmproof}

\textbf{Case 1: $\bm{G_S}$ has exactly 1 non-Gallai endblock
$\bm{B_1}$.}\aside{$B_1$, $G_1$}
Let $G_1$ be the component of $G_S$ containing $B_1$.  Let $B_2$ be a Gallai
endblock in some other component $G_2$.\aside{$B_2$, $G_2$}

\textbf{Case 1.1: $\bm{B_2}$ is $\bm{K_4}$.} 
Now $G_2+x$ contains a good cycle, for
some $x\in S$, as follows (so we are done by Lemma~\ref{good-cycle-lem}).
If some vertex $x$ of $S$ sends at least 2 edges to $B_2$, then $B_2+x$ is
non-Gallai, since $x$ is also non-adjacent to some vertex in
$B_2$ (this includes the case that $G_2=B_2$). Thus, $x+B_2$ contains a good cycle. Otherwise, each vertex of $S$
sends one edge to $B_2$ and some vertex $x$ sends an edge to another endblock of
$G_2$.  Now $x+G_2$ is 2-connected and irregular (so it is non-Gallai and contains a good cycle). 

\textbf{Case 1.2: $\bm{B_2}$ is $\bm{K_3}$.}  
Suppose that no
vertex of $S$ sends edges to exactly 2 vertices of $B_2$.  Now $G_2\ne B_2$, two
vertices of $S$ each send a single edge to $B_2$, and at least one of them, call
it $x$, sends an edge to another endblock of $G_2$.  Now $G_2+x$ is 2-connected
and irregular; thus, it contains a good cycle. So we assume some vertex $x\in
S$ sends exactly two edges to $B_2$. Now $B_2+x$ is non-Gallai and contains a
good cycle.  Since $B_1$ is non-Gallai, we are done by
Lemma~\ref{good-cycle-lem}.

\textbf{Case 1.3: $\bm{B_2}$ is $\bm{K_2}$.}
By symmetry, we assume that every Gallai endblock in $G_2$ is $K_2$.
If $G_2=K_2$, then $G$ contains a copy of $K_4^+$ (with its non-adjacent
3-vertices in $S$), so we are done by Lemma~\ref{K4p-lem}.  If $G_2$ contains
at least 3 endblocks, then each non-cut-vertex in each of these endblocks is
adjacent to all of $S$, so $G$ contains an induced $K_{2,3}$ (with two vertices
in $S$), which is a bipartite theta graph, so we are done by
Lemma~\ref{bipartite-lem}.  Thus, $G_2$ has exactly two endblocks.  If
$G_2=P_3$, then some $x$ in $S$ is adjacent to all vertices in $G_2$, so
$G_2+x$ contains a good cycle, and we are done by Lemma~\ref{good-cycle-lem}.
If some interior block $B_3$
of $G_2$ is non-Gallai (so contains a good
cycle), then each $x$ in $S$ has a path to $B_3$ in $G_2$ (and sends at
most one edge to $B_3$), so we are done by considering $B_1$, $B_3$, and a
shortest path between them.  Instead, assume that every interior block of $G_2$
is Gallai.  Since $G_2$ only has two endblocks, some $x\in S$ sends an
edge to the interior block $B_3$.  Again, $G_2+x$ contains a good cycle.

\textbf{Case 2: $\bm{G_S}$ has 0 non-Gallai endblocks.}

\textbf{Case 2.1: $\bm{G_S}$ has at least 4 endblocks.}
As in the proof of Lemma~\ref{connectivity3-main-lem}, each endblock sends at
least $k-1=3$ edges to $S$.  So the number of endblocks in $G_S$ is at most
$|S|k/(k-1)=3\cdot4/3=4$.  Suppose that $G_S$ has 4 Gallai endblocks. 
If some endblock is a $K_3$, then it sends $S$ at least 4 edges, so $S$ has too
many incident edges (as above), a contradiction.  
Now $G_S$ has only 2 components, and no component is a single endblock; otherwise,
$S$ sends that component at least 4 edges, and $4+3\cdot3>4|S|$, a contradiction.
So let $B_1, B_2$ be endblocks of $G_1$ and let $B_3,B_4$ be 
endblocks of $G_2$, with $B_1,B_2,B_3,B_4\in\{K_2,K_4\}$. 
Similarly, $S$ sends edges only to non-cut-vertices
of $B_1,B_2,B_3,B_4$. 
This implies that if some vertex $x\in S$ sends two edges each to $G_1$ and
$G_2$, then $x+G_1$ and $x+G_2$ each contain a good cycle, and we are done by
Lemma~\ref{good-cycle-lem}. That is because, for each $i\in[2]$, either some
endblock of $G_i$ is $K_4$ or both endblocks of $G_i$ are $K_2$ and there is an
interior endblock $B_5$ that is irregular. Now by Pigeonhole, some $x$ in
$S$ must send two edges each to $G_1$ and $G_2$, so we are done.

\textbf{Case 2.2: $\bm{G_S}$ has a component that is neither $\bm{K_3}$ nor $\bm{K_4}$.}
By Case 2.1, $G_S$ has at most 3 endblocks.
So let $B_1$ be a block of $G_S$ that is a whole component of $G_1$. 
If $B_1$ is an odd cycle of length 5 or more, then $B_1$ sends $S$ at least 10
edges, and $G_S-B_1$ sends $S$ at least $k-1=3$ edges, but $10+3>4|S|=12$, a
contradiction.
If $B_1=K_2$, then every vertex of $S$ is adjacent to all of $B_1$, and some pair
of vertices in $S$ is non-adjacent.  Thus, $G$ contains a $K_4^+$, and we are
done by Lemma~\ref{K4p-lem}.  

Instead assume that $B_1\in\{K_3,K_4\}$.  This implies that some vertex
$x\in S$, sends at least 2 edges to $B_1$ and is not adjacent to all of $B_1$.  
If $x$ sends an edge to two endblocks in $G_2$, then $G_2+x$ contains a good
cycle (either some endblock of $G_2$ is not $K_2$, or else $G_2=P_3$, or $G_2$
contains an interior block that is not $K_2$).
Thus, $x$ sends an edge to at most one endblock of $G_2$.
So if $G_2$ contains two endblocks, call them $B_2$ and $B_3$, then we may
assume that $x$ sends no edges to $B_3$ and that $B_3\neq K_2$. This implies
that some vertex $y\in S$ sends at least 2 edges to $B_3$ and is not adjacent
to all of $B_3$, so $B_3+y$ contains a good cycle.  Now we are done, by
considering the good cycles in $B_1+x$ and in $B_3+y$.
Hence $G_2$ has a single endblock; that is, $G_2=B_2$.
By the previous paragraph, $G_2\ne K_2$, a contradiction.

\textbf{Case 2.3: $\bm{G_S}$ has a component that is $\bm{K_3}$ or has 3 components.}
Suppose that $B_1=B_2=K_3$. If some vertex $x\in S$ sends exactly 2 edges to each
$G_i$, then each $G_i+x$ contains a {good cycle}, so we are done by
Lemma~\ref{good-cycle-lem}.   But such $x$
must exist because (a) each $B_i$ gets 3 edges from $S$, (b) $S$ sends out at
most 12 edges, and (c) we cannot write 6 as a sum of 3 terms, each 1 or 3.
Assume instead that $B_1=K_4$.  If $G_S$ has a third component $G_3$,
then $G_3$ is a single block $B_3$, and counting edges from $S$ shows that 
$B_1=B_2=B_3=K_4$ and $S$ is an independent set.  Now $G$ contains a bipartite
theta graph (take 2 vertices in $S$ and a neighbor of each in each component). 
Assume instead that $G_S$ has only 2 components. 

By symmetry, we assume that $B_1=K_4$ and $B_2=K_3$.
If some vertex $x_i\in S$ sends two edges to each of $B_1$ and $B_2$, then
we are done; so assume this does not happen.  By symmetry, we assume that
$d_{B_2}(x_1)=3$, $d_{B_2}(x_2)=2$, and $d_{B_2}(x_3)=1$, so
$d_{B_1}(x_1)=d_{B_1}(x_2)=1$ and $d_{B_1}(x_3)=2$;  
see Figure~\ref{K3&K4}.
This implies that
$x_2x_3\in E(G)$.  But now $x_2$ and $x_3$ lie in a 4-cycle with two vertices of
$B_1$, and they also lie in a 4-cycle with two vertices of $B_2$.  The union of
these 4-cycles is a bipartite theta graph; so we are done, by
Lemma~\ref{bipartite-lem}.

\begin{figure}[t!]
\begin{center}
\begin{subfigure}{0.5\textwidth}
\centering
\begin{tikzpicture}[rotate=90]
\tikzstyle{uStyle}=[shape = circle, minimum size = 5.5pt, inner sep = 0pt,
outer sep = 0pt, draw, fill=white]
\tikzstyle{lStyle}=[shape = circle, minimum size = 5.5pt, inner sep = 0pt,
outer sep = 0pt, draw=none, fill=none]
\tikzset{every node/.style=uStyle}
\draw[draw=none] (0,0.95) circle (.5mm); 
\draw[thick, white] (0,-4) ellipse (2.2cm and 1cm); 
\draw[thick] (0,-2) ellipse (1.5cm and 0.7cm); 
\draw[thick] (0.5,0.5) node (1) {} (1) -- ++(0,-1) node (2) {} (2) -- ++(-1,0) node (3) {} (3) -- ++(0,1) node (4) {} (3) -- (1) -- (4) -- (2); 
\draw[thick] (0.5,-3.5) node (a) {} (a) -- ++(-0.5,-1) node (b) {} (b) -- ++(-0.5,1) node (c) {} (a) -- (c); 
\draw[thick] (0,-2) -- (-0.5,-2); 
\draw[thick] (-0.5,-2) node (x2) {} (0,-2) node (x3) {} (0.5,-2) node (x1) {}; 
\draw[thick] (0.9,-2) node[lStyle] {$x_1$};
\draw[thick] (0.3,-1.7) node[lStyle] {$x_3$};
\draw[thick] (-0.9,-2) node[lStyle] {$x_2$};
\draw[thick] (x1) edge[bend right=30] (1) (2) -- (x3) -- (3) (x2) edge[bend left=30] (4) (a) -- (x1) -- (c) (x1) edge[bend left=60] (b) (x2) -- (c) (x3) -- (a) (x2) edge[bend right=60] (b);
\draw[thick] (1.5,0) node[lStyle] {$B_1$} (2,-2) node[lStyle] {$S$} (1.5,-4)
node[lStyle] {$B_2$};
\end{tikzpicture}
\caption{\label{K3&K4}}
\end{subfigure}
\begin{subfigure}{0.4\textwidth}
\centering
\begin{tikzpicture}[rotate=90]
\tikzstyle{uStyle}=[shape = circle, minimum size = 5.5pt, inner sep = 0pt,
outer sep = 0pt, draw, fill=white]
\tikzstyle{lStyle}=[shape = circle, minimum size = 5.5pt, inner sep = 0pt,
outer sep = 0pt, draw=none, fill=none]
\tikzset{every node/.style=uStyle}

\draw[draw=none] (0,0.95) circle (.5mm); 
\draw[thick, white] (0,-4) ellipse (2.2cm and 1cm); 
\draw[thick] (0,-2) ellipse (1.5cm and 0.7cm); 
\draw[thick] (0.5,0.5) node (1) {} (1) -- ++(0,-1) node (2) {} (2) -- ++(-1,0) node (3) {} (3) -- ++(0,1) node (4) {} (3) -- (1) -- (4) -- (2); 
\draw[thick] (0.5,-3.5) node (a) {} (a) -- ++(0,-1) node (b) {} (b) -- ++(-1,0) node (c) {} (c) -- ++(0,1) node (d) {} (c) -- (a) -- (d) -- (b); 
\draw[thick] (-0.5,-2) -- (0,-2) -- (0.5,-2); 
\draw[thick] (-0.5,-2) node (x2) {} (0,-2) node (x3) {} (0.5,-2) node (x1) {}; 
\draw[thick] (0.9,-2) node[lStyle] {$x_1$};
\draw[thick] (0.2,-1.7) node[lStyle] {$x_3$};
\draw[thick] (-0.9,-2) node[lStyle] {$x_2$};
\draw[thick] (x1) -- (2) (x1) edge[bend right=30] (1) (x3) -- (3) (x2) edge[bend left=30] (4) (x3) -- (a) (x1) edge[bend left=30] (b) (x2) -- (d) (x2) edge[bend right=30] (c);
\draw[thick] (1.5,0) node[lStyle] {$B_1$} (2,-2) node[lStyle] {$S$} (1.5,-4)
node[lStyle] {$B_2$};
\end{tikzpicture}
\caption{\label{two-K4}}
\end{subfigure}%
\end{center}
\caption{$S$ and the components of $G_S$ when 
(a) $B_1=K_4$ and $B_2=K_3$, and 
(b) $B_1=B_2=K_4$.\label{k5-3conn-B}}
\end{figure}
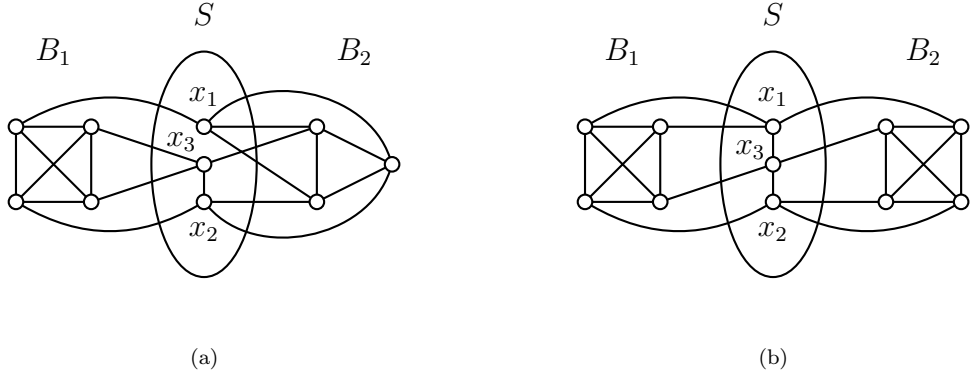

\textbf{Case 2.4: $\bm{G_S}$ has 2 components and both are $\bm{K_4}$.}
Suppose that $B_1=B_2=K_4$ (and $G_S$ has only two components).
Denote $S$ by $x_1,x_2,x_3$.  By symmetry (and counting edges), we assume that
$d_{B_1}(x_1) = d_{B_2}(x_2) = 2$ and
$d_{B_2}(x_1) = d_{B_1}(x_2) = d_{B_1}(x_3) = d_{B_2}(x_3)=1$.
This implies that $x_1x_3,x_2x_3\in E(G)$ and $x_1x_2\notin E(G)$; see Figure~\ref{two-K4}.
If $y_1$ is a neighbor of $x_1$ in $B_1$ and $y_2$ is a neighbor of $x_3$ in
$B_1$, then $\{x_1,x_3,y_1,y_2\}$ induces a 4-cycle in $V(B_1)\cup S$.
Similarly, there exists an induced 4-cycle in $V(B_2)\cup S$ that uses edge
$x_2x_3$.  Each of these 4-cycles is a good cycle, and they intersect in a
single vertex $x_3$.  So we are done.
\end{proof}

Now we handle the case when $G$ is 3-regular and has connectivity 3,
and two vertices of $G$ have distinct lists.

\begin{lem}
Let $G$ be 3-connected and 3-regular, but not $K_4$. Let
$L$ be a 3-assignment for $G$. If there exist $v,w_1\in V(G)$ with 
$L(w_1)\ne L(v)$, then $G$ is $L$-swappable.
\label{connectivity3-3reg-lem}
\end{lem}

\begin{proof}
Since $G$ is connected, we assume $vw_1\in E(G)$.  
Since $G$ is not a clique, $G$ has an $L$-coloring by the list version of
Brooks' Theorem.  Thus, $\LL\neq\emptyset$. Denote $N(v)$ by $\{w_1,w_2,w_3\}$.
If possible, also choose $v$ so that $L(w_2)\neq L(v)$ or $L(w_3)\neq L(v)$ (or
both); this choice is used only
once, near the end of the proof.  If $w_1w_2,w_2w_3\in E(G)$, then
$\{w_1,w_3\}$ is a vertex cut, which contradicts that $G$ is 3-connected. 
Thus, by symmetry, $\{w_1,w_2,w_3\}$ induces at most one edge.

We note that $G-e$ is degree-choosable for each $e\in E(G)$, as follows. Assume
instead that $G-e$ is a Gallai tree for some $e\in E(G)$. Since $G-e$ is
2-connected, it is is a Gallai block, i.e., a complete graph or an odd cycle.
But this is impossible since $G-e$ is irregular. Thus, $G-e$ is
degree-choosable. By Lemma~\ref{overlap-lem}, this implies that $|L(x)\cap
L(y)|\geq2$ for all $xy\in E(G)$. In particular, $|L(w_i)\cap L(v)|\geq2$ for
each $i\in [3]$.

Let $\AA_1:=\cup_{\alpha\in L(w_1)\cap
L(w_2)}\LL_{w_1,\alpha}\cap\LL_{w_2,\alpha}$\aside{$\AA_1$, $\BB_1$}
\aaside{$\AA_2$, $\BB_2$}{4mm}\aaside{$\AA_3$, $\BB_3$}{8mm}
and $\BB_1:=\cup_{\alpha\in L(w_1)\setminus L(v)}\LL_{w_1,\alpha}$. Define
$\AA_2$, $\BB_2$, $\AA_3$, and $\BB_3$
analogously, as in Table~\ref{partition}. 
Let $\AA:=\cup_{i=1}^3\AA_i$ and $\BB:=\cup_{i=1}^3\BB_i$.\aaside{$\AA$,
$\BB$}{7mm}
Observe that $\LL=\AA\cup \BB$. Moreover, by Lemma~\ref{missing-lem}, $\BB_1$
is nonempty and mixes. Recall from above that $G[w_1,w_2,w_3]$ has at most one edge.

\begin{table}[!h]

\begin{center}
\scalebox{1.3}{
\begin{tabular}{c|c}

$\AA$ & $\BB$ \\
\hline
$\AA_1:\cup_{\alpha\in L(w_1)\cap L(w_2)}\LL_{w_1,\alpha}\cap\LL_{w_2,\alpha}$ & $\BB_1:\cup_{\alpha\in L(w_1)\setminus L(v)}\LL_{w_1,\alpha}$ \\
$\AA_2:\cup_{\alpha\in L(w_1)\cap L(w_3)}\LL_{w_1,\alpha}\cap\LL_{w_3,\alpha}$ & $\BB_2:\cup_{\alpha\in L(w_2)\setminus L(v)}\LL_{w_2,\alpha}$ \\ 
$\AA_3:\cup_{\alpha\in L(w_2)\cap L(w_3)}\LL_{w_2,\alpha}\cap\LL_{w_3,\alpha}$ & $\BB_3:\cup_{\alpha\in L(w_3)\setminus L(v)}\LL_{w_3,\alpha}$ \\

\end{tabular}
}
\end{center}
\caption{Every $L$-coloring is in $\AA$ or $\BB$.\label{partition}}
\end{table}

To help clarify the arguments in this proof, we will often draw an auxiliary
graph that has a vertex for each $\AA_i$ and $\BB_j$ that is non-empty.  (Since
$|L(w_i)\cap L(v)|\ge 2$ for all $i$, we have $|L(w_i)\setminus L(v)|\le 1$.
Thus, each nonempty $\BB_j$ mixes.)  If $\AA_i\cup \BB_j$ mixes, then we draw an
edge between the vertices $\AA_i$ and $\BB_j$.  So, to show that $\LL$ mixes, it
suffices to show that this auxiliary graph is connected.

\textbf{Case 1: $\bm{G[w_1,w_2,w_3]}$ has one edge.} By symmetry between $w_2$
and $w_3$, assume that
$w_1w_3\notin E(G)$. So either $w_2w_3\in E(G)$ or $w_1w_2\in E(G)$. 

First suppose $w_2w_3\in E(G)$. This implies that $\AA_3=\emptyset$. Since
$w_1w_2\notin E(G)$, the sets $\AA_1$ and $\AA_2$ are both nonempty. By
Lemma~\ref{common-lem}(2), $\AA_1\cup \BB_1$ and $\AA_2\cup \BB_1$ both mix;
see Figure~\ref{lem15-fig1}a.  Moreover, for each $i\in\{2,3\}$, if $\BB_i$ is
nonempty, then $\BB_i\cup\AA_{i-1}$ mixes by Lemma~\ref{common-lem}(2). So
$\AA\cup\BB$ mixes; that is, $\LL$ mixes. 

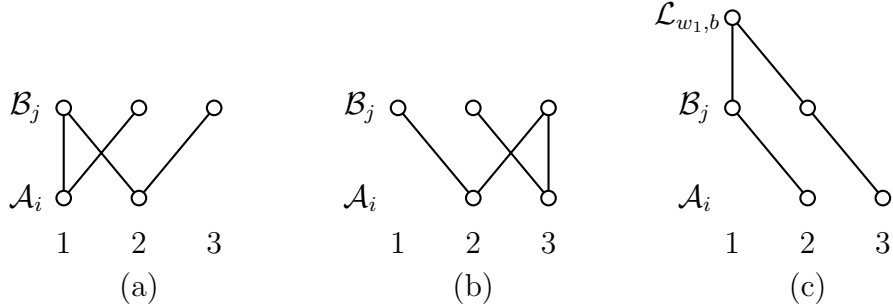
\begin{figure}[!h]
\centering
\begin{tikzpicture}[yscale=-1.2, thick]
\tikzstyle{uStyle}=[shape = circle, minimum size = 5.5pt, inner sep = 0pt,
outer sep = 0pt, draw, fill=white]
\tikzstyle{lStyle}=[shape = circle, minimum size = 5.5pt, inner sep = 0pt,
outer sep = 0pt, draw=none, fill=white]
\tikzset{every node/.style=uStyle}

\draw (1,-1) node (A1) {} (1,-2) node (B1) {};
\draw (2,-1) node (A2) {} (2,-2) node (B2) {};
\draw (3,-2) node (B3) {};

\draw (1,-0.5) node[lStyle] {1};
\draw (2,-0.5) node[lStyle] {2};
\draw (3,-0.5) node[lStyle] {3};
\draw (0.5,-1) node[lStyle] {$\AA_i$};
\draw (0.5,-2) node[lStyle] {$\BB_j$};
\draw (B3) -- (A2) -- (B1) -- (A1) -- (B2);

\draw (2,0) node[lStyle] {(a)};

\begin{scope}[xshift=1.75in]
\draw (1,-2) node (B1) {};
\draw (2,-1) node (A2) {} (2,-2) node (B2) {};
\draw (3,-1) node (A3) {} (3,-2) node (B3) {};

\draw (1,-0.5) node[lStyle] {1};
\draw (2,-0.5) node[lStyle] {2};
\draw (3,-0.5) node[lStyle] {3};
\draw (0.5,-1) node[lStyle] {$\AA_i$};
\draw (0.5,-2) node[lStyle] {$\BB_j$};
\draw (B1) -- (A2) -- (B3) -- (A3) -- (B2);

\draw (2,0) node[lStyle] {(b)};

\end{scope}

\begin{scope}[xshift=3.5in]
\draw (1,-2) node (B1) {};
\draw (2,-1) node (A2) {} (2,-2) node (B2) {};
\draw (3,-1) node (A3) {}; 
\draw (1,-3) node (C1) {};

\draw (1,-0.5) node[lStyle] {1};
\draw (2,-0.5) node[lStyle] {2};
\draw (3,-0.5) node[lStyle] {3};
\draw (0.5,-1) node[lStyle] {$\AA_i$};
\draw (0.5,-2) node[lStyle] {$\BB_j$};
\draw (0.4,-3) node[lStyle,shape=rectangle] {$\LL_{w_1,b}$};
\draw (A2) -- (B1) -- (C1) -- (B2) -- (A3);

\draw (2,0) node[lStyle] {(c)};

\end{scope}

\end{tikzpicture}
{
\captionsetup{width=.8\textwidth}
\caption{(a) The case that $w_2w_3\in E(G)$.
(b) The case that $w_1w_2\in E(G)$ and $\BB_3\ne\emptyset$.
(c) The case that $w_1w_2\in E(G)$, $\BB_3=\emptyset$, and
$\BB_2\ne\emptyset$.
\label{lem15-fig1}}
}
\end{figure}

Instead assume $w_1w_2\in E(G)$. So $\AA_1=\emptyset$. Since $w_2w_3\notin
E(G)$, the sets $\AA_2$ and $\AA_3$ are both nonempty. By
Lemma~\ref{common-lem}(2), $\AA_2\cup\BB_1$ mixes. Now we show that
$\BB_2=\BB_3=\emptyset$. Suppose first that $\BB_3\neq\emptyset$. By
Lemma~\ref{common-lem}(2), $\AA_2\cup \BB_3$ and $\AA_3\cup \BB_3$ both mix; 
see Figure~\ref{lem15-fig1}b.  If
$\BB_2=\emptyset$, then we are done, since $\AA\cup\BB$ mixes. So assume
$\BB_2\neq\emptyset$. Now by Lemma~\ref{common-lem}(2) $\AA_3\cup \BB_2$ mixes.
Thus, $\AA\cup \BB$ mixes; that is, $\LL$ mixes. So we assume $\BB_3=\emptyset$. 

Now suppose that $\BB_2\neq\emptyset$. By Lemma~\ref{common-lem}, $\BB_2\cup
\AA_3$ mixes. Since $\BB_1$ and $\BB_2$ are nonempty, there exists $\alpha\in
L(w_1)\setminus L(v)$ and $\beta\in L(w_2)\setminus L(v)$. If
$\alpha\neq\beta$, then $\BB_1\cap \BB_2\neq\emptyset$; thus, $\BB_1\cup \BB_2$
mixes. So $\AA\cup\BB$ mixes, and we are done. So assume $\alpha=\beta$. If
$L(w_1)=L(w_2)$, then there exists $\gamma\in L(v)\setminus(L(w_1)\cap
L(w_2))$. By Lemma~\ref{missing-lem}, the set $\LL_{v,\gamma}$ mixes. 
Moreover,
$\BB_1\cap\LL_{v,\gamma}$ and $\BB_2\cap\LL_{v,\gamma}$ are both nonempty.
Thus, $\BB_1\cup\LL_{v,\gamma}\cup \BB_2$ mixes. So $\AA\cup\BB$ mixes, and we
are done. So assume $L(w_1)\neq L(w_2)$. Pick $a\in L(v)\cap L(w_1)\cap
L(w_2)$ and $b\in L(w_1)\setminus L(w_2)$ and $c\in L(w_2)\setminus L(w_1)$;
see Figure \ref{edge+2d}. Note that $\LL_{w_1,b}$ mixes by
Lemma~\ref{missing-lem}, and $\BB_2\cap\LL_{w_1,b}\neq\emptyset$. So
$\BB_2\cup\LL_{w_1,b}$ mixes; 
see Figure~\ref{lem15-fig1}c.
Moreover, there exists
$\vph\in\LL_{w_1,b}\cap\LL_{w_3,b}$, and $\vph$ mixes with $\BB_1$ by
Lemma~\ref{common-lem}(2). Since $\vph\in\LL_{w_1,b}$, the set $\LL_{w_1,b}\cup
\BB_1$ mixes; hence, $\BB_1\cup \BB_2$ mixes. So $\AA\cup\BB$ mixes, and we
are done. So we instead assume $\BB_2=\emptyset$. 

\begin{figure}[h!]
\centering
\begin{subfigure}{0.4\textwidth}
\centering
\begin{tikzpicture}[every node/.style={scale=0.8}]
\draw[thick] (0,1) -- (-1.5,-1) -- (0,-1) -- cycle (0,1) -- (1.5,-1);
\draw[thick] (0,1) node[uStyle] {$v$}; 
\draw[thick] (-1.5,-1) node[uStyle] {$w_1$};
\draw[thick] (0,-1) node[uStyle] {$w_2$};
\draw[thick] (1.5,-1) node[uStyle] {$w_3$};
\draw[thick] (0,1.5) node {$abc$};
\draw[thick] (-1.5,-1.5) node {$abd$};
\draw[thick] (0,-1.5) node {$acd$};
\draw[thick] (1.5,-1.5) node {$abc$};
\end{tikzpicture}
\caption{}
\label{edge+2d}
\end{subfigure}%
\begin{subfigure}{0.4\textwidth}
\centering
\begin{tikzpicture}[every node/.style={scale=0.8}]
\draw[thick] (0,1) -- (-1.5,-1) -- (0,-1) -- cycle (0,1) -- (1.5,-1);
\draw[thick] (0,1) node[uStyle] {$v$}; 
\draw[thick] (-1.5,-1) node[uStyle] {$w_1$};
\draw[thick] (0,-1) node[uStyle] {$w_2$};
\draw[thick] (1.5,-1) node[uStyle] {$w_3$};
\draw[thick] (0,1.5) node {$abc$};
\draw[thick] (-1.5,-1.5) node {$abd$};
\draw[thick] (0,-1.5) node {$abc$};
\draw[thick] (1.5,-1.5) node {$abc$};
\end{tikzpicture}
\caption{}
\label{edge+1d}
\end{subfigure}
\caption{Two cases when $w_1w_2\in E(G)$. (a) A 3-assignment for $N[v]$ when
$d\in L(w_1)\cap L(w_2)$ and $L(w_1)\neq L(w_2)$ and $L(v)=L(w_3)$. (b) A
3-assignment for $N[v]$ when $L(v)=L(w_2)=L(w_3)$.}
\end{figure}
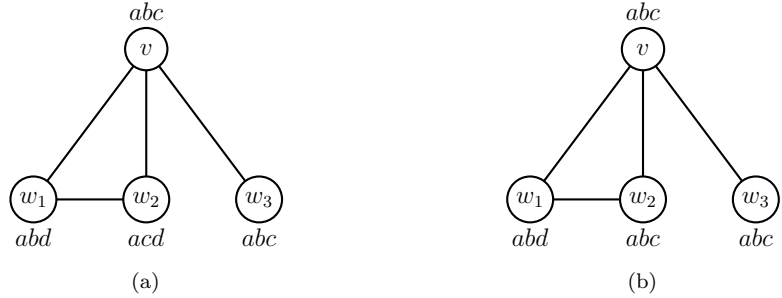

Now $\BB_2=\emptyset=\BB_3$, so
$L(w_2)=L(w_3)=L(v)$; see Figure
\ref{edge+1d}. From above $\AA_1=\emptyset$ and $\AA_2\cup \BB_1$
mixes.  So it suffices to show that $\AA_3$ mixes with $\BB_1$.  By
Lemma~\ref{missing-lem}, we know that $\LL_{w_2,c}$ and $\LL_{w_1,d}$ and
$\LL_{v,c}$ are nonempty sets that mix.  Note that $\BB_1=\LL_{w_1,d}$.  Since
$\LL_{w_1,d}\cap \LL_{w_2,c}\ne\emptyset$, we see that $\LL_{w_2,c}\cup
\LL_{w_1,d}$ mixes.  Similarly, $\LL_{w_1,d}\cap \LL_{v,c}\ne\emptyset$, so
$\LL_{w_1,d}\cup \LL_{v,c}$ mixes.  Finally $\LL_{v,c}\cap
\LL_{w_2,\alpha}\cap\LL_{w_3,\alpha}\ne \emptyset$ for each $\alpha\in\{a,b\}$.
Thus,
$\LL_{v,c}\cup\bigcup_{\alpha\in\{a,b\}}(\LL_{w_2,\alpha}\cap\LL_{w_3,\alpha})$
mixes.  Combining all these observations gives that $\LL_{w_2,c}\cup
\LL_{w_1,d}\cup\LL_{v,c}\cup\bigcup_{\alpha\in\{a,b\}}(\LL_{w_2,\alpha}\cap\LL_{w_3,\alpha})$
mixes.  But this set contains $\AA_3\cup \BB_1$, so we are done.

\textbf{Case 2: $\bm{G[w_1,w_2,w_3]}$ has no edges.} So $\AA_1,\AA_2$, and
$\AA_3$ are all nonempty. Moreover, $\BB_1\cup \AA_1$ and $\BB_1\cup \AA_2$
both mix by Lemma~\ref{common-lem}(2). If $\BB_2\neq\emptyset$, then $\BB_2\cup
\AA_1$ and $\BB_2\cup \AA_3$ both mix by Lemma~\ref{common-lem}(2); see
Figure~\ref{lem15-fig2}a. If, in
addition, $\BB_3=\emptyset$, then $\AA\cup\BB_1\cup\BB_2$ mixes, and we are
done. Otherwise, $\BB_3\neq\emptyset$. So $\BB_3\cup \AA_2$ and $\BB_3\cup
\AA_3$ both mix by Lemma~\ref{common-lem}(2). Thus, $\AA\cup\BB$ mixes, and we
are done. So assume $\BB_2=\emptyset$; by symmetry, also $\BB_3=\emptyset$. 

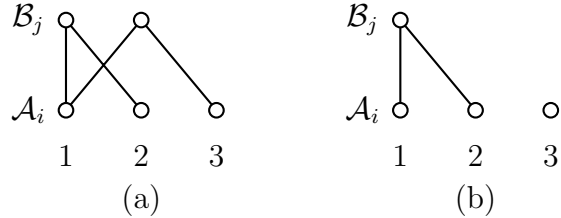
\begin{figure}[!h]
\centering
\begin{tikzpicture}[yscale=-1.2, thick]
\tikzstyle{uStyle}=[shape = circle, minimum size = 5.5pt, inner sep = 0pt,
outer sep = 0pt, draw, fill=white]
\tikzstyle{lStyle}=[shape = circle, minimum size = 5.5pt, inner sep = 0pt,
outer sep = 0pt, draw=none, fill=white]
\tikzset{every node/.style=uStyle}

\draw (1,-1) node (A1) {} (1,-2) node (B1) {};
\draw (2,-1) node (A2) {} (2,-2) node (B2) {};
\draw (3,-1) node (A3) {};

\draw (1,-0.5) node[lStyle] {1};
\draw (2,-0.5) node[lStyle] {2};
\draw (3,-0.5) node[lStyle] {3};
\draw (0.5,-1) node[lStyle] {$\AA_i$};
\draw (0.5,-2) node[lStyle] {$\BB_j$};
\draw (A3) -- (B2) -- (A1) -- (B1) -- (A2);

\draw (2,0) node[lStyle] {(a)};

\begin{scope}[xshift=1.75in]
\draw (1,-1) node (A1) {} (1,-2) node (B1) {};
\draw (2,-1) node (A2) {}; (2,-2) node (B2) {};
\draw (3,-1) node (A3) {}; (3,-2) node (B3) {};

\draw (1,-0.5) node[lStyle] {1};
\draw (2,-0.5) node[lStyle] {2};
\draw (3,-0.5) node[lStyle] {3};
\draw (0.5,-1) node[lStyle] {$\AA_i$};
\draw (0.5,-2) node[lStyle] {$\BB_j$};
\draw (A1) -- (B1) -- (A2);

\draw (2,0) node[lStyle] {(b)};

\end{scope}

\end{tikzpicture}

{
\captionsetup{width=.6\textwidth}
\caption{Here $G[w_1,w_2,w_3]$ has no edges.  (a) The case that
$\BB_2\ne\emptyset$.
(b) The case that $\BB_2=\emptyset=\BB_3$.
\label{lem15-fig2}}
}
\end{figure}

Now it suffices to show that $\AA_3\cup \BB_1$ mixes. Fix $z_1$ and $z_2$ in
$N(w_1)\setminus\{v\}$. By symmetry between $w_1$ and $v$, we assume $N(w_1)$
induces no edges and $L(w_1)=L(z_1)=L(z_2)$; see Figure~\ref{noedge+noalpha}.
By Lemma~\ref{missing-lem}, $\LL_{v,c}$ is nonempty and mixes. Now we show that
$\AA_3$ mixes. Form $G'$ from $G$ by deleting $v$ then identifying $w_2$ and
$w_3$. Call the new vertex $w_{23}$. Note that $G'$ is 2-connected;
equivalently, $w_{23}$ is not a cut-vertex in $G'$. To see this, note that
$\{w_2,w_3\}$ is not a vertex cut in $G$, since $G$ is 3-connected. So
$\{w_2,w_3,v\}$ is not a vertex cut in $G$, since $v$ is a leaf in
$G-\{w_2,w_3\}$. But the components of $G-\{v,w_2,w_3\}$ are the same as those
of $G'-w_{23}$. So $w_{23}$ is not a cut-vertex in $G'$, as desired. 

Let $L'$
be a 3-assignment for $G'$ with $L'(w_{23})=L(w_2)$ and $L'(x)=L(x)$ for all
other $x\in V(G')$. Now $G'$ is $L'$-swappable, by Lemma~\ref{extra-lem}(a) with
$w:=w_1$, since $d_{G'}(w_1)=2<3=|L'(w_1)|$. We note that every coloring
$\vph'$ of $G'$ can be extended to
a coloring $\vph$ in $G$ since $|\cup_{x\in N(v)}\vph'(x)|\le2$. Thus,
$L'$-colorings are in bijection with colorings in $\AA_3$. Moreover, every
$\alpha,\beta$-swap performed in $\vph'$ can also be performed in $\vph$ as
follows. If the swap does not involve $N(v)$, or
$\vph(v)\notin\{\alpha,\beta\}$, then we perform the same swap as in $\vph'$
(possibly a swap at $w_2$ and $w_3$ each). If $\vph(w_1)=\vph(w_2)=\vph(w_3)$
or $\beta=d$ (so the swap is at $w_1$), then we can recolor $v$ with
$\gamma\notin\{\alpha,\beta\}$ then perform the swap. Otherwise, suppose
$\vph(v)=\alpha$ and (i) $\vph(w_1)=\beta$, or (ii)
$\vph(w_2)=\vph(w_3)=\beta$. In case (i), the swap is valid since $\{a,b\}\in
L(v)$, and in case (ii), the swap is valid since $L(v)=L(w_2)=L(w_3)$. Thus,
$\AA_3$ mixes, as claimed.

For every $\vph\in\cup_{\gamma\in\{a,b\}}\LL_{w_2,\gamma}\cap\LL_{w_3,\gamma}$,
either (i) $\vph(w_1)=\vph(w_2)=\vph(w_3)$, or (ii) $\vph(w_1)=d$, or (iii)
$\vph(v)=c$.  In case (i), we have $\AA_1\cap\AA_2\cap\AA_3\ne \emptyset$, so
we are done.  In case (ii), $\vph\in \BB_1$; and in case (iii), $\vph \in \LL_{v,c}$.
Thus, $\LL_{w_2,\gamma}\cap\LL_{w_3,\gamma}$ mixes either with $\BB_1$ or
$\LL_{v,c}$, for each $\gamma\in\{a,b\}$. So it suffices to show that $\BB_1$
mixes with $\LL_{v,c}$. By our choice of $v$ at the beginning of the proof,
there exists $u\in N(z_1)\setminus\{w_1\}$ with $L(u)=L(z_1)$; otherwise, we
would have chosen $z_1$ instead of $v$. Since $G$ is 3-connected, $G-z_1-z_2$
contains a $u,v$-path $P$. Note that $L(x)\neq L(y)$ for some $x$ and $y$ that
are successive on $P$. By construction, either $x$ or $y$ is not in $N(v)\cup
N(w_1)$; see Figure \ref{noedge+nod}. By symmetry, assume $x\notin N(v)\cup
N(w_1)$. Since there exists $\gamma\in L(x)\setminus L(y)$, by
Lemma~\ref{missing-lem}, the set $\LL_{x,\gamma}$ mixes. Further,
$\LL_{x,\gamma}\cap \BB_1\neq\emptyset$ and
$\LL_{x,\gamma}\cap\LL_{v,c}\neq\emptyset$; thus, $\BB_1\cup\LL_{v,c}$ mixes,
and we are done.
\end{proof}

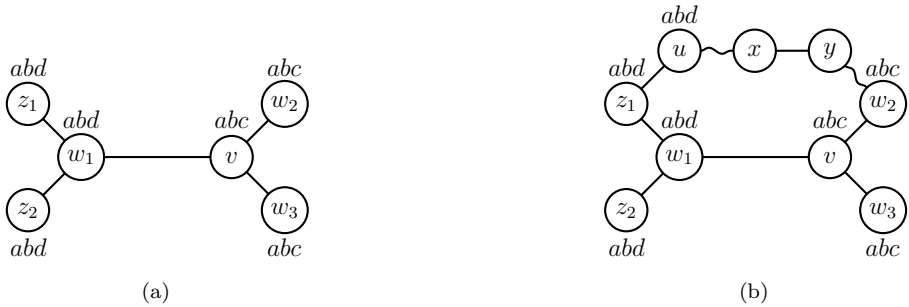
\begin{figure}[h!]
\begin{center}
\begin{subfigure}{0.3\textwidth}
\centering
\begin{tikzpicture}[every node/.style={scale=0.8}]
\draw[thick] (2,0) node[uStyle] (v) {$v$};
\draw[thick] (0,0) node[uStyle] (w1) {$w_1$};
\draw[thick] (v) -- ++(45:1) node[uStyle] (w2) {$w_2$};
\draw[thick] (v) -- ++(315:1) node[uStyle] (w3) {$w_3$};
\draw[thick] (w1) -- ++(135:1) node[uStyle] (z1) {$z_1$};
\draw[thick] (w1) -- ++(225:1) node[uStyle] (z2) {$z_2$};
\draw[thick] (w1) -- (v); 
\foreach \i/\j in {2.7/1.2, 2.7/-1.2, 2/0.5}
\draw[thick] (\i,\j) node {$abc$};
\foreach \i/\j in {-0.7/1.2, -0.7/-1.2, 0/0.5}
\draw[thick] (\i,\j) node {$abd$};
\draw[thick, white] (0,1.9) node {$abd$};
\end{tikzpicture}
\caption{\label{noedge+noalpha}}
\end{subfigure}\hspace{3cm}%
\begin{subfigure}{0.3\textwidth}
\centering
\begin{tikzpicture}[every node/.style={scale=0.8}]
\draw[thick] (2,0) node[uStyle] (v) {$v$};
\draw[thick] (0,0) node[uStyle] (w1) {$w_1$};
\draw[thick] (v) -- ++(45:1) node[uStyle] (w2) {$w_2$};
\draw[thick] (v) -- ++(315:1) node[uStyle] (w3) {$w_3$};
\draw[thick] (w1) -- ++(135:1) node[uStyle] (z1) {$z_1$};
\draw[thick] (w1) -- ++(225:1) node[uStyle] (z2) {$z_2$};
\draw[thick] (z1) -- ++(45:1) node[uStyle] (u) {$u$};
\draw[thick] (w2) ++(135:1) node[uStyle] (y) {$y$};
\draw[thick] (y) -- ++(-1,0) node[uStyle] (x) {$x$};
\draw[thick] (w1) -- (v); 
\draw[thick] (y) edge[decorate, decoration={snake, amplitude=0.4mm}] (w2); 
\draw[thick] (u) edge[decorate, decoration={snake, amplitude=0.4mm}] (x); 
\foreach \i/\j in {2.7/1.2, 2.7/-1.2, 2/0.5}
\draw[thick] (\i,\j) node {$abc$};
\foreach \i/\j in {-0.7/1.2, -0.7/-1.2, 0/0.5}
\draw[thick] (\i,\j) node {$abd$};
\draw[thick] (0,1.9) node {$abd$};
\end{tikzpicture}
\caption{\label{noedge+nod}}
\end{subfigure}
\caption{(a) A 3-assignment for $N(v)\cup N(w_1)$ when $L(w_1)=L(z_1)=L(z_2)$
and $L(v)=L(w_2)=L(w_3)$. (b) The first instance, along a $u,v$-path, of a
consecutive pair $x,y$ with distinct lists.}
\end{center}
\end{figure}

\begin{thm}
Let $G$ be $3$-regular with connectivity 3. 
If $L$ is a 3-assignment for $G$, then $G$ is $L$-swappable unless either (a)
$G=K_4$ or (b) $G= K_2\,\square\, K_3$ and $L(v)=L(w)$ for all $v,w\in V(G)$. 
\label{connectivity3}
\end{thm}

\begin{proof}
Let $L$ be a 3-assignment for $G$. If there exists $v,w\in V(G)$ such that
$L(v)\neq L(w)$, then $G$ is $L$-swappable by
Lemma~\ref{connectivity3-3reg-lem} (or Lemma~\ref{clique-lem} when
$G=K_4$)\footnote{Although
we have not yet proved this lemma, its proof is independent of the present
theorem, so invoking it now is logically consistent. We make this choice to
preserve the narrative flow of Section~\ref{main-proof-sec}.}.
Otherwise, $L(v)=L(w)$ for every $v,w\in V(G)$. 
If $G\ne K_2 \, \square\, K_3$, then
by \hyperref[FJP Theorem]{Theorem~A}~\cite{FJP}, $G$ is $L$-swappable.
\end{proof}

\begin{lem}
Let $G$ be 4-connected, $k$-regular, but not a clique. Let $L$ be a $k$-assignment for $G$.
If there exist $v,w_1\in V(G)$ with $vw_1\in E(G)$ and $L(w_1)\ne L(v)$, then $G$ is
$L$-swappable.
\label{4connected-lem}
\end{lem}
\begin{proof}
By the list version of Brooks' Theorem~\cite{ERT}, $G$ has an $L$-coloring; that is,
$\LL\ne\emptyset$.  
Denote the neighbors of $v$ by $w_1, \ldots, w_k$.
By Lemma~\ref{missing-lem}, for each $i\in[k]$ and each $\alpha\in
L(w_i)\setminus L(v)$, the set $\LL_{w_i,\alpha}$ mixes.
By Lemma~\ref{common-lem}(2), for all distinct $j,\ell\in[k]$ and each $\beta\in L(w_j)\cap
L(w_{\ell})$, the set $\LL_{w_j,\beta}\cap\LL_{w_{\ell},\beta}$ mixes.
Let $\LL_1:=\cup_{i\in[k]}\cup_{\alpha\in L(w_i)\setminus L(v)}\LL_{w_i,\alpha}$.\aside{$\LL_1$, $\LL_2$}
Let $\LL_2:=\cup_{j,\ell\in[k], j\ne\ell}\cup_{\beta\in L(w_j)\cap
L(w_{\ell})}\LL_{w_j,\beta}\cap \LL_{w_{\ell},\beta}$.
By Pigeonhole, for every $\vph\in\LL$, either $\vph\in \LL_1$ or $\vph\in
\LL_2$, or both.  So $\LL=\LL_1\cup \LL_2$.

We will often want to $L$-color a small set of vertices, $S$, and show that our
coloring $\vph$ of $S$ extends to an $L$-coloring of $G$.  This is possible whenever
$S\subset N(v)$ and $|S|\le 3$ and $|(\cup_{x\in S}\vph(x))\cap L(v)| < |S|$.
After coloring $S$, we greedily color $G-S$ in order of nonincreasing
distance from $v$.  This uses that
$G-S$ is connected, since $G$ is 4-connected and $|S|<4$. And the same
argument shows that all such $L$-colorings (with a fixed coloring of $S$) mix.

\textbf{Case 1: $\bm{|L(w_1)\setminus L(v)|\ge 2}$.}
Note, for each $\alpha\in L(w_1)\setminus L(v)$, that
$\LL_{w_1,\alpha}\ne\emptyset$.  Further, $\cup_{\alpha\in L(w_1)\setminus
L(v)}\LL_{w_1,\alpha}$ mixes, by Lemma~\ref{missing-lem}.
Suppose there exists $i\in [k]\setminus \{1\}$ with $L(w_i)\setminus L(v)\ne
\emptyset$.  Similar to above, $\cup_{\beta\in L(w_i)\setminus L(v)}
\LL_{w_i,\beta}$ mixes.  Further, given $\beta\in L(w_i)\setminus L(v)$, there
exists $\alpha\in L(w_1)\setminus L(v)$ such that $\alpha\ne \beta$.  Since
$\LL_{w_1,\alpha}\cap\LL_{w_i,\beta}\ne\emptyset$, we see that $\LL_1$ mixes.

Fix $i,j\in[k]$ with $w_iw_j\notin E(G)$ and $L(w_i)\cap L(w_j)\ne \emptyset$; fix $\alpha\in
L(w_i)\cap L(w_j)$.  Recall that $\LL_{w_i,\alpha}\cap\LL_{w_j,\alpha}$ mixes,
by Lemma~\ref{common-lem}(1).
Suppose that $i,j\in[k]\setminus\{1\}$.  Fix $\alpha\in
L(w_i)\cap L(w_j)$.  Now there exists $\beta\in
L(w_1)\setminus(L(v)\cup\{\alpha\})$.  Thus, there exists an $L$-coloring
$\vph$ with $\vph(w_i)=\vph(w_j)=\alpha$ and $\vph(w_1)=\beta$, by the
remark before Case 1.  Note that $\vph\in\LL_{w_1,\beta}\cap
(\LL_{w_i,\alpha}\cap \LL_{w_j,\alpha})$.  
Thus, $\LL_{w_i,\alpha}\cap\LL_{w_j,\alpha}$ mixes with $\LL_1$ for each such 
choice of $i,j$, and $\alpha\in L(w_i)\cap L(w_j)$.

Finally, suppose there exists $i\in[k]\setminus\{1\}$ with $\alpha\in
L(w_1)\cap L(w_i)$ and $w_1w_i\notin E(G)$.  Let $G':=G-vw_1$.  Let
$L'(w_i):=\{\alpha\}$,
$L'(w_1)=\{\alpha\}\cup(L(w_1)\setminus L(v))$, and let $L'(x):=L(x)$
for all other $x\in V(G')$.  Note that $G'$ has an $L'$-coloring and all
$L'$-colorings of $G'$ mix, by Lemma~\ref{degen-lem}.  Further, the
$L'$-colorings $\vph'$ of $G'$ are in bijection with $L$-colorings $\vph$ of
$G$ with $\vph(w_i)=\alpha$
and $\vph(w_1)\in\{\alpha\}\cup(L(w_1)\setminus L(v))$.  This latter set
includes a coloring in $\LL_{w_1,\alpha}\cap\LL_{w_i,\alpha}$ and also a
coloring in $\LL_1$.
Thus, $\LL_1\cup \LL_2$ mixes; that is, $\LL$ mixes.

By symmetry among the $w_i$'s, we henceforth assume that $|L(w_i)\setminus L(v)|\le
1$ for all $i\in[k]$.

\textbf{Case 2: $\bm{|L(w_i)\setminus L(v)|\le 1}$ for all $\bm{i\in[k]}$ and
$\bm{|(L(w_1)\cup L(w_2))\setminus L(v)|\ge 2}$.}
As above, $\cup_{\alpha\in L(w_i)\setminus L(v)}\LL_{w_i,\alpha}$ mixes, for
each $i\in[k]$,  by Lemma~\ref{missing-lem}. By hypothesis, there exist
$\alpha\in L(w_1)\setminus L(v)$ and $\beta\in L(w_2)\setminus L(v)$ with
$\alpha\ne \beta$.  So there exists $\vph\in \LL_{w_1,\alpha}\cap
\LL_{w_2,\beta}$.  Further, for every $i\in[k]$ with $L(w_i)\ne L(v)$,
there exists $\gamma\in L(w_i)\setminus L(v)$ such that either (a)
$\LL_{w_1,\alpha}\cap \LL_{w_i,\gamma}\ne \emptyset$ (and $\gamma\ne \alpha$)
or (b) $\LL_{w_2,\beta}\cap \LL_{w_i,\gamma}\ne \emptyset$ (and $\gamma\ne
\beta$).  Thus, $\LL_1$ mixes.

Now instead fix any distinct $i,j\in[k]$ with
$w_iw_j\notin E(G)$.  
For each $\alpha\in L(w_i)\cap L(w_j)$, recall that
$\LL_{w_i,\alpha}\cap\LL_{w_j,\alpha}$ mixes by Lemma~\ref{common-lem}(1). 
For each such $i,j$ there exists $\alpha\in L(w_i)\cap L(w_j)$
and such a coloring $\vph$ either with $\vph(w_1)\notin L(v)$ or with
$\vph(w_2)\notin L(v)$, unless $\{i,j\}=\{1,2\}$.  However, in this case,
$\bigcup_{\alpha\in L(w_1)\cap L(w_2)} \LL_{w_1,\alpha}\cap
\LL_{w_2,\alpha}\cup\bigcup_{\beta\in L(w_1)\setminus L(v)}\LL_{w_1,\beta}$ mixes, by
Lemma~\ref{common-lem}(2).
Thus, $\LL_1\cup \LL_2$ mixes; that is, $\LL$ mixes.  

By symmetry among the $w_i$'s, we now assume $|(L(w_i)\cup L(w_j))\setminus L(v)|\le
1$ for all $i,j\in[k]$.

\textbf{Case 3: $\bm{|\cup_{i\in[k]}L(w_i)\setminus L(v)|=1}$.}
Let $\{\alpha\}$ denote $\cup_{i\in[k]}L(w_i)\setminus L(v)$.
Note that $\LL_{w_h,\alpha}$ mixes whenever $h\in[k]$ and $\alpha\in L(w_h)$,
by Lemma~\ref{missing-lem}.  Since $G$ is not a clique, there exist $i,j\in[k]$
such that $w_iw_j\notin E(G)$.  For each such $i,j$ and $\beta\in L(w_i)\cap
L(w_j)$, the set $\LL_{w_i,\beta}\cap \LL_{w_j,\beta}$ mixes, by
Lemma~\ref{common-lem}(1).  Now we show that $\LL_{w_h,\alpha}$ mixes with
$\LL_{w_i,\beta}\cap\LL_{w_j,\beta}$ for all such $h, i, j, \alpha$, and
$\beta$.  First suppose that $\alpha\ne \beta$.  If $h\notin\{i,j\}$, then there
exists an $L$-coloring $\vph$ with $\vph(w_h)=\alpha$ and
$\vph(w_i)=\vph(w_j)=\beta$.  If $h\in\{i,j\}$, then this claimed mixing follows from
Lemma~\ref{common-lem}(2).  So assume instead that $\alpha=\beta$.  If
$h\in\{i,j\}$, then this is trivial, since
$\LL_{w_i,\beta}\cap\LL_{w_j,\beta}\subseteq \LL_{w_h,\alpha}$.
So assume $h\notin\{i,j\}$.  Pick $\beta'\in L(w_i)\cap
L(w_j)\setminus\{\alpha\}$.
There exists an $L$-coloring $\vph$ with $\vph(w_h)=\alpha$ and
$\vph(w_i)=\vph(w_j)=\beta'$.  But now we are again done by
Lemma~\ref{common-lem}(2).
More specifically, $\LL_{w_h,\alpha}$ mixes and $\LL_{w_i,\alpha}$ mixes.  By
Lemma~\ref{common-lem}(2), also
$\LL_{w_i,\beta}\cap(\LL_{w_j,\beta}\cup\LL_{w_j,\alpha})$ mixes.
Since $\LL_{w_i,\alpha}\cap\LL_{w_j,\alpha}\subseteq \LL_{w_i,\alpha}$, we are
done.
\end{proof}

The next lemma handles the case that $G=K_{k+1}$.  
The main ideas in the proof are similar to those in the previous
proof, but the details are a bit different.

\begin{lem}
Let $G=K_{k+1}$, where $k\ge 3$, and let $L$ be a $k$-assignment.
If there exist $v,w\in V(G)$ such that $L(v)\ne L(w)$, then $G$ is
$L$-swappable.
\label{clique-lem}
\end{lem}
\begin{proof}
We denote the vertices of $G$ by $v_1,\ldots,v_{k+1}$, and we consider three
cases.  

\textbf{Case 1: There exist $\bm{i,j\in [k+1]}$ such that $\bm{|L(v_i)\setminus
L(v_j)|\ge 2}$.}
Fix $\vph\in\LL$.  By symmetry, assume that $i=1$ and $j=2$.
Since $|L(v_2)|=k$ and $|V(G)|=k+1$, there exists $\ell\in[k+1]$ such that
$\vph(v_{\ell})\notin L(v_2)$. Let $\alpha:=\vph(v_{\ell})$.
Let $\LL_1:= \cup_{\beta\in L(v_1)\setminus L(v_2)}\LL_{v_1,\beta}$.
By Lemma~\ref{missing-lem}, 
we know that $\LL_{v_{\ell},\alpha}$ mixes and also $\LL_1$
mixes.  We will construct an $L$-coloring $\vph'$ such that $\vph'\in \LL_1\cap
\LL_{v_{\ell},\alpha}$.  This will prove that $\LL_1$ mixes with
$\LL_{v_{\ell},\alpha}$.  Since $\vph\in \LL_{v_{\ell},\alpha}$ and $\vph$ is
an arbitrary $L$-coloring, we conclude that $G$ is $L$-swappable.
Let $\vph'(v_{\ell})=\alpha$, color $v_1$ from
$L(v_1)\setminus(L(v_2)\cup\{\beta\})$, and thereafter color greedily, finishing
with $v_2$.

\textbf{Case 2: There exist distinct $\bm{i,j,\ell\in [k+1]}$ and distinct colors
$\bm{\alpha_i,\alpha_j}$ such that $\bm{\alpha_i\in L(v_i)\setminus L(v_\ell)}$ and
$\bm{\alpha_j\in L(v_j)\setminus L(v_\ell)}$.}
We can assume that Case 1 does not hold, so $|L(w)\setminus L(x)|\le 1$ for all
$w,x\in V(G)$.
By symmetry, we assume that $i=1$, $j=2$, and $\ell=3$.  Let
$\LL_1:=\LL_{v_1,\alpha_1}$ and $\LL_2:=\LL_{v_2,\alpha_2}$.  By Lemma~\ref{missing-lem},
note that $\LL_1$ mixes and $\LL_2$ mixes.  We construct $\vph'\in \LL_1\cap
\LL_2$.  Let $\vph'(v_1):=\alpha_1$, let $\vph'(v_2):=\alpha_2$, and thereafter
color greedily, finishing with $v_3$.  
Fix an arbitrary $L$-coloring $\vph$. 
Now it suffices to show that either $\vph$ is $L$-equivalent to some
$L$-coloring in $\LL_1$ or $\vph$ is $L$-equivalent to some $L$-coloring in
$\LL_2$.  Since $|L(v_3)|=k$ and $|V(G)|=k+1$, there exists $h$ such that
$\vph(v_h)\notin L(v_3)$.  Further, either $\vph(v_h)\ne \alpha_1$ or
$\vph(v_h)\ne \alpha_2$; by symmetry, assume the former.  Now, as in the
previous case, there exists an $L$-coloring $\vph'$ such that $\vph'\in
\LL_1\cap \LL_{v_h,\vph(v_h)}$.  Since $\LL_{v_h,\vph(v_h)}$ mixes, by
Lemma~\ref{missing-lem}, 
we are done.

\textbf{Case 3: There exist $\bm{i,j,\ell\in [k+1]}$ with $\bm{i\notin\{j,\ell\}}$ and a
color $\bm{\alpha}$ such that $\bm{\alpha\in (L(v_j)\cap L(v_\ell))\setminus
L(v_i)}$.} We can assume neither Case 1 nor Case 2 holds. So, $|L(v_i)\cap
L(v_j)|=|L(v_i)\cap L(v_\ell)|=k-1$. If $j=\ell$, i.e., there is only one vertex
$v_j$ such that $\alpha\in L(v_j)$ (implying that $L(v_l)=L(v_i)$ for all $l\in
[k+1]\setminus\{j\}$), then every coloring $\vph$ lies in $\LL_{v_j,\alpha}$.
This is because $|L(v_i)|=k$ and $|V(G)|=k+1$, which implies that some vertex
(namely $v_j$) has $\vph(v_j)\notin L(v_i)$. By Lemma~\ref{missing-lem},
$L_{v_j,\alpha}$ mixes. So, we conclude that $G$ is $L$-swappable. Thus, we
assume that $j\neq \ell$. 

By symmetry, assume that $i=1, j=2,$ and $\ell=3$. If $L(v_2)=L(v_3)$, then
pick $\beta\in L(v_1)\setminus L(v_2)$. By Lemma~\ref{missing-lem}, each of
$\LL_{v_1,\beta}, \LL_{v_2,\alpha},$ and $\LL_{v_3,\alpha}$ mixes. Moreover,
$\LL_{v_1,\beta}\cap\LL_{v_2,\alpha}\neq\emptyset$ (color $v_1$ with $\beta$,
color $v_2$ with $\alpha$, then color greedily, finishing with $v_3$). So,
$\LL_{v_1,\beta}$ mixes with $\LL_{v_2,\alpha}$. By symmetry between $v_2$ and
$v_3$, the set $\LL_{v_1,\beta}$ also mixes with $\LL_{v_3,\alpha}$. Further,
if $\alpha\in L(v_h)$ for some $h\in [k+1]\setminus\{2,3\}$, then
$\LL_{v_h,\alpha}$ mixes by Lemma~\ref{missing-lem}. And as above,
$\LL_{v_h,\alpha}\cap\LL_{v_1,\beta}\neq\emptyset$, which implies that
$\LL_{v_h,\alpha}$ mixes with $\LL_{v_1,\beta}$. Since $|L(v_1)|=k$ and
$|V(G)|=k+1$, for every $L$-coloring $\vph$ there exists $h\in [k+1]$ such
that $\vph(v_h)\notin L(v_1)$. In particular, $\vph(v_h)=\alpha$. So,
$\vph\in\LL_{v_h,\alpha}$, and we conclude that $G$ is $L$-swappable. Thus,
we assume $L(v_2)\neq L(v_3)$.

By symmetry, assume $L(v_1)=[k], L(v_2)=(\{\alpha\}\cup [k])\setminus\{1\},$ and
$L(v_3)=\{\alpha\}\cup [k-1]$. By Lemma~\ref{missing-lem}, each of
$\LL_{v_1,k},\LL_{v_1,1},\LL_{v_2,\alpha},\LL_{v_2,k},\LL_{v_3,\alpha},$ and
$\LL_{v_3,1}$ mixes. Moreover, $\LL_{v_1,k}\cap\LL_{v_2,\alpha}\neq\emptyset$
(color greedily, finishing with $v_3$). Similarly,
$\LL_{v_1,k}\cap\LL_{v_3,1}\neq\emptyset,
\LL_{v_1,1}\cap\LL_{v_2,k}\neq\emptyset,$ and
$\LL_{v_1,1}\cap\LL_{v_3,\alpha}\neq\emptyset$. Thus, $\LL_{v_1,k}$ mixes
with each of $\LL_{v_2,\alpha}$ and $\LL_{v_3,1}$; similarly, $\LL_{v_1,1}$
mixes with each of $\LL_{v_2,k}$ and $\LL_{v_3,\alpha}$. By coloring greedily,
finishing with $v_1$, there exists $\vph\in\LL_{v_2,\alpha}\cap\LL_{v_3,2}$.
And a $2,\alpha$-swap in $\vph$ at $v_2$ gives a coloring in
$\LL_{v_3,\alpha}$. Thus, $\LL_{v_2,\alpha}$ mixes with $\LL_{v_3,\alpha}$.
Finally, if $\alpha\in L(v_l)$ for some $l\in[k+1]\setminus\{2,3\}$, then
$\LL_{v_l,\alpha}$ mixes by Lemma~\ref{missing-lem}. And, $\LL_{v_l,\alpha}$
mixes with $\LL_{v_1,1}$ since $\LL_{v_l,\alpha}\cap\LL_{v_1,1}\neq\emptyset$.
As above, for every $L$-coloring $\vph$, there exists $h\in [k+1]$ such that
$\vph(v_h)=\alpha$.  Thus, we conclude that $G$ is $L$-swappable.   
\end{proof}

\begin{lem}
\label{noW4-coloring-lem}
Let $G$ be a 4-connected graph that is $k$-regular, but does not contain an
induced 4-wheel, $W_4$, and is not $K_{k+1}$.  If $L(v)=[k]$ for all $v\in V(G)$,
then $G$ is $L$-swappable.
\end{lem}
\begin{proof}
Let $G$ satisfy the hypothesis.  Fix an arbitrary vertex $v\in V(G)$ and denote
$N(v)$ by $\{w_1,\ldots,w_k\}$.  By Pigeonhole, for every $\vph\in
\LL$ there exist distinct $i,j,\alpha\in[k]$ with $w_iw_j\notin
E(G)$ and $\vph(w_i)=\vph(w_j)=\alpha$.  Since $L(x)=[k]$ for all $x\in V(G)$,
for all distinct $i,j$ such that $w_iw_j\notin E(G)$ and all distinct
$\alpha,\beta\in[k]$, the sets $\LL_{w_i,\alpha}\cap\LL_{w_j,\alpha}$ and
$\LL_{w_i,\beta}\cap\LL_{w_j,\beta}$ mix with each other; we simply use
$\alpha,\beta$-swaps at $w_i$ and $w_j$.  So, in this proof,\aside{$\LL_{i,j}$} 
let $\LL_{i,j}:=\cup_{\alpha\in[k]}(\LL_{w_i,\alpha}\cap\LL_{w_j,\alpha})$.
For convenience, when $w_iw_j\in E(G)$, let $\LL_{i,j}:=\emptyset$.
So $\LL=\cup_{i,j\in[k], i\ne j}\LL_{i,j}$.

By Lemma~\ref{common-lem}, if $w_iw_j\notin E(G)$, then $\LL_{i,j}$ mixes.
Fix distinct $i,j,\ell\in[k]$ such that $w_iw_j,w_jw_{\ell}\notin E(G)$.  We
show that $\LL_{i,j}$ mixes with $\LL_{j,{\ell}}$.  Let $G':=G-vw_j$,
let $L'(w_i):=\{1\}$, $L'(w_{\ell})=\{2\}$, $L'(w_j):=\{1,2\}$, and
$L'(x):=[k]$ for all $x\in V(G)-\{w_i,w_j,w_{\ell}\}$.  Let
$G'':=G-\{w_i,w_j,w_{\ell}\}$.  Note that $G''$ is connected, since $G$ is
4-connected.  Let $\sigma''$ be a vertex order of $V(G'')$ by non-increasing
distance from $v$, and let $\sigma':=w_i,w_{\ell},w_j,\sigma''$.  Now $\sigma'$
shows that $\LL'$ mixes, since each vertex $x$ is preceded by fewer than
$|L'(x)|$ neighbors in $\sigma'$.

Now consider distinct $h,i,j,\ell\in[k]$ such that $w_hw_i,w_jw_{\ell}\notin
E(G)$.  We must show that $\LL_{h,i}$ mixes with $\LL_{j,\ell}$.  
If $w_hw_j\notin E(G)$, then $\LL_{h,i}$ mixes with $\LL_{h,j}$ and $\LL_{h,j}$
mixes with $\LL_{j,\ell}$, as above, so we are done.  Assume instead that
$w_hw_j\in E(G)$.  By symmetry, also $w_iw_j,w_iw_{\ell},w_{\ell}w_h\in E(G)$. 
However, now $G$ contains an induced 4-wheel, a contradiction.
\end{proof}

Now we combine the previous four lemmas to completely handle the 4-connected
case.

\begin{thm}
Let $G$ be 4-connected.
If $G$ is $k$-regular and not $K_{k+1}$, then $G$ is
$k$-swappable.  If $G=K_{k+1}$, $L$ is a $k$-assignment, and $L$ is not
identical everywhere, then $G$ is $L$-swappable.
\label{connectivity4}
\end{thm}
\vspace{-.3in}

\begin{proof}
If $G$ contains $W_4$ as an induced subgraph, then we are done by
Lemma~\ref{W4-lem} and Lemma~\ref{H-lem}.  So assume it does not.  Fix a
$k$-assignment $L$ for $G$.  If there exist $x,y\in V(G)$ such that $L(x)\ne
L(y)$, then $G$ is $L$-swappable, by Lemma~\ref{4connected-lem} or
Lemma~\ref{clique-lem}.  So assume instead that $L(x)=L(y)$ for all $x,y\in
V(G)$.  By symmetry, we assume that $L(v)=[k]$ for all $v\in V(G)$.  Now we are
done by Lemma~\ref{noW4-coloring-lem}.
\end{proof}

\section*{Acknowledgment}
Thanks to a referee for carefully reading this manuscript and making numerous
suggestions that helped improve the presentation.

\scriptsize{

}

\end{document}